\documentclass[titlepage]{amsart} 						
\usepackage[dvips]{graphicx}						
\usepackage{amssymb}	
\usepackage{amsmath}					
\usepackage{epstopdf}
\usepackage{eucal}
\theoremstyle{plain}							
\newtheorem{Lem}{Lemma}[section]
\newtheorem{Thm}[Lem]{Theorem}
\newtheorem{Prop}[Lem]{Proposition}
\newtheorem{Cor}[Lem]{Corollary}
\theoremstyle{definition}
\newtheorem*{Def}{Definition}

\usepackage[rflt]{floatflt}
\usepackage{wrapfig}
\usepackage{capt-of}
\usepackage{graphicx}
\usepackage{multicol}
\usepackage{subcaption}
\usepackage[usenames,dvipsnames]{xcolor}
\usepackage{url}
\usepackage{amsaddr}
\begin{document}

\author{Jordan Awan}
\address{Dept. of Statistics, Purdue University}
\author{Claire Frechette}
\address{Mathematics Dept., University of Minnesota Twin Cities}
\author{Yumi Li}
\address{Willamette University}
\author{Elizabeth McMahon}
\address{Mathematics Dept., Lafayette College, Easton, PA}
\email{mcmahone@lafayette.edu}

\title{Demicaps in AG(4,3) and Their Relation to Maximal Cap Partitions}
\thanks{Research supported by NSF grant DMS-1063070.}


\date{}


\begin{abstract}
In this paper, we introduce a fundamental substructure of maximal caps in the affine geometry $AG(4,3)$ that we call \emph{demicaps}.  Demicaps provide a direct link to particular partitions of $AG(4,3)$ into 4 maximal caps plus a single point.  The full collection of 36 maximal caps that are in exactly one partition with a given cap $C$ can be expressed as unions of two disjoint demicaps taken from a set of 12 demicaps; these 12 can also be found using demicaps in $C$.  The action of the affine group on these 36 maximal caps includes actions related to the outer automorphisms of $S_6$.
\end{abstract}

\maketitle

\section{Introduction}\label{S:Intro}


We introduce a new structure called a demicap in $AG(4,3)$, which is an extremely well-studied finite geometry.  A demicap is a 10-point configuration that is a fundamental building block for the maximal caps in the geometry.   In addition,  demicaps provide a direct link to partitions of $AG(4,3)$ into disjoint maximal caps plus a distinguished point.  Thus, demicaps provide valuable information about maximal caps and partitions. 

$AG(4,3)$ can be partitioned into 4 disjoint maximal caps plus a distinguished point.  Those partitions are in two equivalence classes under affine transformations.  One of those classes contains two pairs of maximal caps that appear in only one partition.  These are the partitions we focus on in this paper.  If $C$ is a maximal cap, then there are exactly 36 maximal caps disjoint from $C$  that appear in exactly one partition with $C$, and demicaps can be used to find these 36 caps. Further, any transformation of $AG(4,3)$ that fixes $C$ as a set necessarily permutes the 36 maximal caps, and those permutations act in the same way as an outer autmorphism of $S_6$.



The problem of determining the size and structure of maximal caps in $AG(n,3)$  has received a good deal of attention from many researchers.  Many mathematicians are very interested in the ``cap-set problem,'' the problem of finding bounds on the size of the maximal caps in $AG(n,3)$.  In 2016, new upper bounds on the size of such sets were found by Ellenberg and Gijswijt  \cite{MR3583358}.  A more complete understanding of the geometric structure of maximal caps could have implications for cap size calculations in other  dimensions.  The structures we introduce and examine in this paper could assist with this project.

\medskip
We now define the basic structures we will be examining in this paper.  
A {\em cap} in the affine geometry $AG(n,3)$ is a set of points containing no lines; a {\em maximal cap} is a cap of largest possible size.  A cap is  {\em complete} if it is not a subset of a larger cap.  There are caps in $AG(n,3), \, n \ge 3$, which are complete but not maximal.

The finite geometry $AG(n,3)$ is an $n$-dimensional vector space over $\mathbb{Z}_3$; three points are on a line if and only if the corresponding vectors sum to $\vec0$. The full transformation group of $AG(n,3)$  is the affine group $\mathit{Aff}(n,3) = GL(n,3) \ltimes \mathbb{Z}_3^n$, where  $(A,\vec{b})$ in the semidirect product corresponds to the transformation $\vec{v}\mapsto A\vec{v}+\vec{b}$.  

All maximal caps in the geometry $AG(4,3)$ have 20 points (G.\ Pellegrino, \cite{MR0363952}), and all are affinely equivalent (R.\  Hill, \cite{MR695829}).  Each maximal cap $C$ in $AG(4,3)$ has a unique {\em anchor point $a$}, where the 20 points of $C$ consist of 10 pairs of points, so that each pair makes a line with $a$.  

Partitions of affine and projective space into caps have been examined by several researchers (see, for example,  \cite{MR2005528} and \cite{MR718954}). In this paper, we will focus on partitions of $AG(4,3)$ into four disjoint maximal caps together with their common anchor point, which were first explored in \cite{MR3262358}.  This gives a partition of $AG(4,3)$ into five caps, the smallest number possible; these particular partitions also have the largest possible number of maximal caps.

\smallskip
In \cite{MR3262358}, it was shown that  any pair of disjoint maximal caps in $AG(4,3)$ must necessarily have the same anchor point; a pair of disjoint maximal caps appear as blocks in either 1, 2 or 6 distinct partitions, so such pairs are called {\em 1-completable, 2-completable, or 6-completable pairs}, respectively.
The affine group $\mathit{Aff}(4,3)$ acts transitively on 1-completable pairs, transitively on 2-completable pairs and transitively on 6-completable pairs.  In this paper, we concentrate on 1-completeable pairs. 

Any maximal cap $C$ has exactly 36  maximal caps disjoint from $C$ that appear with $C$ in precisely one partition of $AG(4,3)$. Additionally, $C$ can be written as the union of two disjoint demicaps in 36 ways. We show that these two sets are connected by a bijection.  
Not only are the demicap partitions of a maximal cap $C$  related to the maximal caps in exactly one  partition with $C$, but one can also use the demicaps to find the other pair of maximal caps in the partition.  

The structure and number of demicaps is explored in Section \ref{S:demics}.  In Section \ref{S:results}, we give the correspondence between a pair of demicaps whose union is $C$ and one maximal cap $C'$ making a 1-completable pair with $C$. We then use the same pair of demicaps to determine the other two maximal caps in the partition containing $C$ and $C'$.

In Section \ref{S:36caps}, we show that there are two collections of six demicaps $\{R_1,\dots, R_6\}$ and $\{C_1,\dots, C_6\}$ so that every one of the 36 maximal caps in a unique partition with $C$ is $R_i \cup C_j$ for some $i,j \leq 6$. These demicaps can be found from any pair of disjoint demicaps whose union is $C$.  We finish with an exploration of the full set of 36 maximal caps that are in exactly one partition with a maximal cap $C$.  The subgroup of the affine group that fixes $C$ acts on the set of 36 caps as the full automorphism group of $S_6$, including outer automorphism actions.


Thus, what we will show in this paper is that a demicap, a natural substructure of a maximal cap in $AG(4,3)$, encodes all the information necessary to find one of the 36 maximal caps that are in exactly one partition with that cap, as well as a way to identify the other two maximal caps in that partition.  This provides a full accounting of the partitions of $AG(4,3)$ consisting of two 1-completable pairs of maximal caps.


\section{Partitions of $AG(4,3)$ into disjoint maximal caps plus their common anchor point}\label{S:cpsptns}

Throughout our exposition, we will use lower case letters,  $a$, for points in $AG(4,3)$.  Occasionally, we will want to consider those points as vectors in $GF(3)^4$, in which case we will write the point with vector notation, as $\vec{a}$.

Using the same visualization scheme as is used by Davis and Maclagan \cite{MR2005098},\footnote{The lines in a $3\times3$ subgrid are three points horizontally, vertically, or diagonally, including diagonals as on a torus; the points in all other lines appear in three subgrids that are in the same position as a line in the $3\times 3$ subgrids, so that, when  the three subgrids are superimposed, those points either lie on top of one another, or they are in the same positions as a line in a subgrid. See Figure \ref{F:AG43cap}.} one maximal cap of $AG(4,3)$ is shown below in Figure \ref{F:AG43cap}. (This  is the first maximal cap with anchor point $\vec{0}$ in the lexicographic order.) The anchor point is in the upper left corresponding to $\vec{0}$; it is easily verified that the cap consists of 10 lines through $\vec{0}$.  

\newsavebox{\grid}
\savebox{\grid}(0,78){\setlength{\unitlength}{.02cm}
\multiput(0,94)(0,15){4}{\line(1,0){45}}		
\multiput(0,94)(15,0){4}{\line(0,1){45}}		
\multiput(47,94)(0,15){4}{\line(1,0){45}}	
\multiput(47,94)(15,0){4}{\line(0,1){45}}	
\multiput(94,94)(0,15){4}{\line(1,0){45}}	
\multiput(94,94)(15,0){4}{\line(0,1){45}}	
\multiput(0,47)(0,15){4}{\line(1,0){45}}		
\multiput(0,47)(15,0){4}{\line(0,1){45}}		
\multiput(47,47)(0,15){4}{\line(1,0){45}}	
\multiput(47,47)(15,0){4}{\line(0,1){45}}	
\multiput(94,47)(0,15){4}{\line(1,0){45}}	
\multiput(94,47)(15,0){4}{\line(0,1){45}}	
\multiput(0,0)(0,15){4}{\line(1,0){45}}		
\multiput(0,0)(15,0){4}{\line(0,1){45}}		
\multiput(47,0)(0,15){4}{\line(1,0){45}}	
\multiput(47,0)(15,0){4}{\line(0,1){45}}	
\multiput(94,0)(0,15){4}{\line(1,0){45}}	
\multiput(94,0)(15,0){4}{\line(0,1){45}}	
}

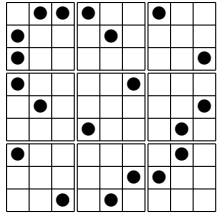
\begin{figure}[h]
\begin{picture}(94,80)\setlength{\unitlength}{.02cm}		%
\usebox{\grid}
\put (22.5,131){\circle* {8}}	
\put (37.5,131){\circle* {8}}	
\put (7.5,116){\circle* {8}}		
\put (7.5,101){\circle* {8}}		
\put (54.5,131){\circle* {8}}	
\put (69.5,116){\circle* {8}}	
\put (101.5,131){\circle* {8}}	
\put (131.5,101){\circle* {8}}	
\put (7.5,84){\circle* {8}}		
\put (22.5,69){\circle* {8}}		
\put (84.5,84){\circle* {8}}		
\put (54.5,54){\circle* {8}}		
\put (131.5,69){\circle* {8}}	
\put (116.5,54){\circle* {8}}	
\put (7.5,37){\circle* {8}}		
\put (37.5,7){\circle* {8}}		
\put (84.5,22){\circle* {8}}		
\put (69.5,7){\circle* {8}}		
\put (116.5,37){\circle* {8}}	
\put (101.5,22){\circle* {8}}	
\end{picture}
\begin{center}
\caption{A maximal cap in $AG(4,3)$ with anchor point in the upper left.}
\label{F:AG43cap}
\end{center}
\end{figure}

In \cite{MR3262358}, it was shown that $AG(4,3)$ can be partitioned into four disjoint maximal caps together with their common anchor point.  One such partition is shown in Figure \ref{F:AG43ptn}.  Any two maximal caps with different anchor points necessarily intersect (Prop. 3.4, \cite{MR3262358}), so all maximal caps in such a partition have the same anchor point, the blank box in the upper left.

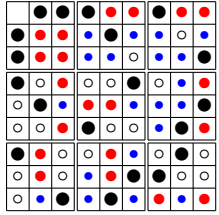
\begin{figure}[h]
\begin{picture}(94,80)\setlength{\unitlength}{.02cm}		%
\usebox{\grid}  %
\put (22.5,131){\circle* {8}}	
\put (37.5,131){\circle* {8}}	
\put (7.5,116){\circle* {8}}		
\put (22.5,116){\textcolor{red}{\circle* {7}}}	
\put (37.5,116){\textcolor{red}{\circle* {7}}}	
\put (7.5,101){\circle* {8}}		
\put (22.5,101){\textcolor{red}{\circle* {7}}}	
\put (37.5,101){\textcolor{red}{\circle* {7}}}	
\put (54.5,131){\circle* {8}}	
\put (69.5,131){\textcolor{red}{\circle* {7}}}	
\put (84.5,131){\textcolor{red}{\circle* {7}}}	
\put (54.5,116){\textcolor{blue}{\circle* {5}}}	
\put (69.5,116){\circle* {8}}	
\put (84.5,116){\textcolor{blue}{\circle* {5}}}	
\put (54.5,101){\textcolor{blue}{\circle* {5}}}	
\put (69.5,101){\textcolor{blue}{\circle* {5}}}	
\put (84.5,101){{\circle {5}}}	
\put (101.5,131){\circle* {8}}	
\put (116.5,131){\textcolor{red}{\circle* {7}}}	
\put (131.5,131){\textcolor{red}{\circle* {7}}}	
\put (101.5,116){\textcolor{blue}{\circle* {5}}}	
\put (116.5,116){{\circle {5}}}	
\put (131.5,116){\textcolor{blue}{\circle* {5}}}	
\put (101.5,101){\textcolor{blue}{\circle* {5}}}	
\put (116.5,101){\textcolor{blue}{\circle* {5}}}	
\put (131.5,101){\circle* {8}}	
%
\put (7.5,84){\circle* {8}}		
\put (22.5,84){{\circle {5}}}		
\put (37.5,84){\textcolor{red}{\circle* {7}}}		
\put (7.5,69){{\circle {5}}}		
\put (22.5,69){\circle* {8}}		
\put (37.5,69){\textcolor{blue}{\circle* {5}}}		
\put (7.5,54){{\circle {5}}}	
\put (22.5,54){{\circle {5}}}		
\put (37.5,54){\textcolor{red}{\circle* {7}}}		
\put (54.5,84){{\circle {5}}}		
\put (69.5,84){{\circle {5}}}		
\put (84.5,84){\circle* {8}}		
\put (54.5,69){\textcolor{red}{\circle* {7}}}		
\put (69.5,69){\textcolor{red}{\circle* {7}}}		
\put (84.5,69){\textcolor{blue}{\circle* {5}}}		
\put (54.5,54){\circle* {8}}		
\put (69.5,54){{\circle {5}}}		
\put (84.5,54){{\circle {5}}}		
\put (101.5,84){{\circle {5}}}	
\put (116.5,84){\textcolor{blue}{\circle* {5}}}	
\put (131.5,84){\textcolor{red}{\circle* {7}}}	
\put (101.5,69){\textcolor{blue}{\circle* {5}}}	
\put (116.5,69){\textcolor{blue}{\circle* {5}}}	
\put (131.5,69){\circle* {8}}	
\put (101.5,54){\textcolor{blue}{\circle* {5}}}	
\put (116.5,54){\circle* {8}}	
\put (131.5,54){\textcolor{red}{\circle* {7}}}	
%
\put (7.5,37){\circle* {8}}		
\put (22.5,37){\textcolor{red}{\circle* {7}}}		
\put (37.5,37){{\circle {5}}}		
\put (7.5,22){{\circle {5}}}		
\put (22.5,22){\textcolor{red}{\circle* {7}}}		
\put (37.5,22){{\circle {5}}}		
\put (7.5,7){{\circle {5}}}		
\put (22.5,7){\textcolor{blue}{\circle* {5}}}		
\put (37.5,7){\circle* {8}}		
\put (54.5,37){{\circle {5}}}		
\put (69.5,37){\textcolor{red}{\circle* {7}}}		
\put (84.5,37){\textcolor{blue}{\circle* {5}}}		
\put (54.5,22){\textcolor{blue}{\circle* {5}}}		
\put (69.5,22){\textcolor{red}{\circle* {7}}}			
\put (84.5,22){\circle* {8}}		
\put (54.5,7){\textcolor{blue}{\circle* {5}}}		
\put (69.5,7){\circle* {8}}		
\put (84.5,7){\textcolor{blue}{\circle* {5}}}		
\put (101.5,37){{\circle {5}}}	
\put (116.5,37){\circle* {8}}	
\put (131.5,37){{\circle {5}}}	
\put (101.5,22){\circle* {8}}	
\put (116.5,22){{\circle {5}}}	
\put (131.5,22){{\circle {5}}}	
\put (101.5,7){\textcolor{red}{\circle* {7}}}		
\put (116.5,7){\textcolor{blue}{\circle* {5}}}		
\put (131.5,7){\textcolor{red}{\circle* {7}}}		
\end{picture}\vspace{-.1in}
\begin{center}
\caption{A partition of $AG(4,3)$ into 4 maximal caps plus their associated anchor point.}
\label{F:AG43ptn}
\end{center}
\end{figure}

Now, we will give further results from \cite{MR3262358} that will be needed for this paper.  The terminology used is given in the following definition.

\begin{Def}
Given a pair of disjoint maximal caps $C$ and $C'$ (necessarily with the same anchor point), if $C$ and $C'$  are in exactly $k$  partitions, then we call $\{C,C'\}$  a {\em $k$-completable pair}.  
\end{Def}

The next proposition shows that the only possible values of $k$ for which two  maximal caps are a $k$-completable pair are $k=1,2,6$. it also gives further information about the structure of the partitions. 

\begin{Prop}[Prop. 3.5, Prop. 3.6, Thm. 3.8, \cite{MR3262358}]\label{P:REU}
Let $C$ be a maximal cap with anchor point $a$; let  $\mathcal{T}$ be the group of linear transformations of $AG(4,3)$ that fix  $C$ as a set, and let $T \in \mathcal{T}$.  

(1) There are  198 maximal caps disjoint from $C$; every one of those together with $C$ can be extended to a partition of $AG(4,3)$ into four maximal caps plus the common anchor point. Of those 198 maximal caps, 36 make a 1-completable pair with $C$; 90 make a 2-completable pair with $C$; and 72 make a 6-completable pair with $C$.  

(2) If $\{C,C'\}$ is an $i$-completable pair for $i=1,2,6$, then so is $\{T(C),T(C)'\}$.  

(3)  Let $\{\{a\},A,B,C,D\}$ be a partition of $AG(4,3)$ into 4 mutually disjoint maximal caps together with their anchor point.  Then the maximal caps can be paired so that exactly one of two possibilities holds:

\begin{itemize}
\item $\{A,B\}$ and $\{C,D\}$ are both 1-completable pairs.
\item $\{A,B\}$ and $\{C,D\}$ are both 2-completable pairs.
\end{itemize}

In both cases, $\{A,C\}$, $\{A,D\}$,  $\{B,C\}$,  and $\{B,D\}$  are 6-completable pairs.

(4)  A 6-completable pair $\{C,C'\}$  appears in exactly one partition with two 1-completable pairs and in five partitions with two 2-completable pairs.

(5) The affine group $\mathit{Aff}(4,3)$ acts on the set of partitions of $AG(4,3)$ into 4 mutually disjoint maximal caps and the associated anchor point $a$. There  are  two orbits under that action: one orbit consists of partitions with two 1-completable pairs, and the other consists of partitions with two 2-completable pairs.  
\end{Prop}

Our goal in this paper is to understand what distinguishes the first of these two orbits geometrically. 


\section{Maximal caps in $AG(4,3)$ and demicaps}\label{S:demics}

Let $C$ be a maximal cap in $AG(n,3)$.  Any point not in $C$ must complete a line with a pair of points from $C$, or $C$ would not be maximal.  How many such lines can a given point complete?  The following proposition answers this question for $AG(n,3)$, $n \leq 6$. (Since our focus is on $AG(4,3)$, that result is stated first.)  .  

\begin{Prop}\label{P:3lines}
A maximal cap $C$ in $AG(4,3)$ consists of 10 pairs of points, each of which completes a line with the anchor point.  Any  point in $AG(4,3)$ other than the anchor point and the points in $C$ completes exactly three lines with pairs of points in $C$.

In $AG(n,3)$, for all other $n$ where the size of a maximal cap is known, a similar result holds. See Table \ref{T:LinesComp}.

\begin{table}[h]
\caption{Known sizes of maximal caps and the number of lines completed by non-anchor points not in the cap.}
\begin{center}
\begin{tabular}{|c|c|c|c|} \hline
Space & Maximal & Anchor? & Number of lines every non- \\
& cap size & & anchor point completes \\ \hline
$AG(2,3)$ & 4 & Yes &  1 \\ 
$AG(3,3)$ & 9 & No &  2 \\ 
$AG(4,3)$ & 20 & Yes &  3 \\ 
$AG(5,3)$ & 45 & No &  5 \\ 
$AG(6,3)$ & 112 & Yes & 10 \\  \hline
\end{tabular}
\end{center}
\label{T:LinesComp}
\end{table}%

\end{Prop} 

\begin{proof}
The result in all cases can be shown by examining one maximal cap, since all maximal caps are affinely equivalent in every dimension where the size is known (\cite{MR695829}, \cite{MR1911459}, \cite{MR2372838}), and incidence structures are preserved by affine transformations.  
We can demonstrate this proposition for $AG(4,3)$ using the maximal cap $C$ pictured in Figure \ref{F:AG43cap}. A program called the Cap Builder \cite{Capbuilder}  allows a user to enter points in $AG(n,3)$ for $n=2, 3,\dots,7$ that include no lines; the program also counts the  number of lines completed using points from in the cap for each point not in the cap.  In Figure \ref{F:Slines}, you see the results for a maximal cap in $AG(4,3)$, where the letter $A$ signifies that 10 lines are completed by that point, the anchor point.
\end{proof}

\begin{figure}[htbp]
\begin{center}
\includegraphics[width=1.6in]{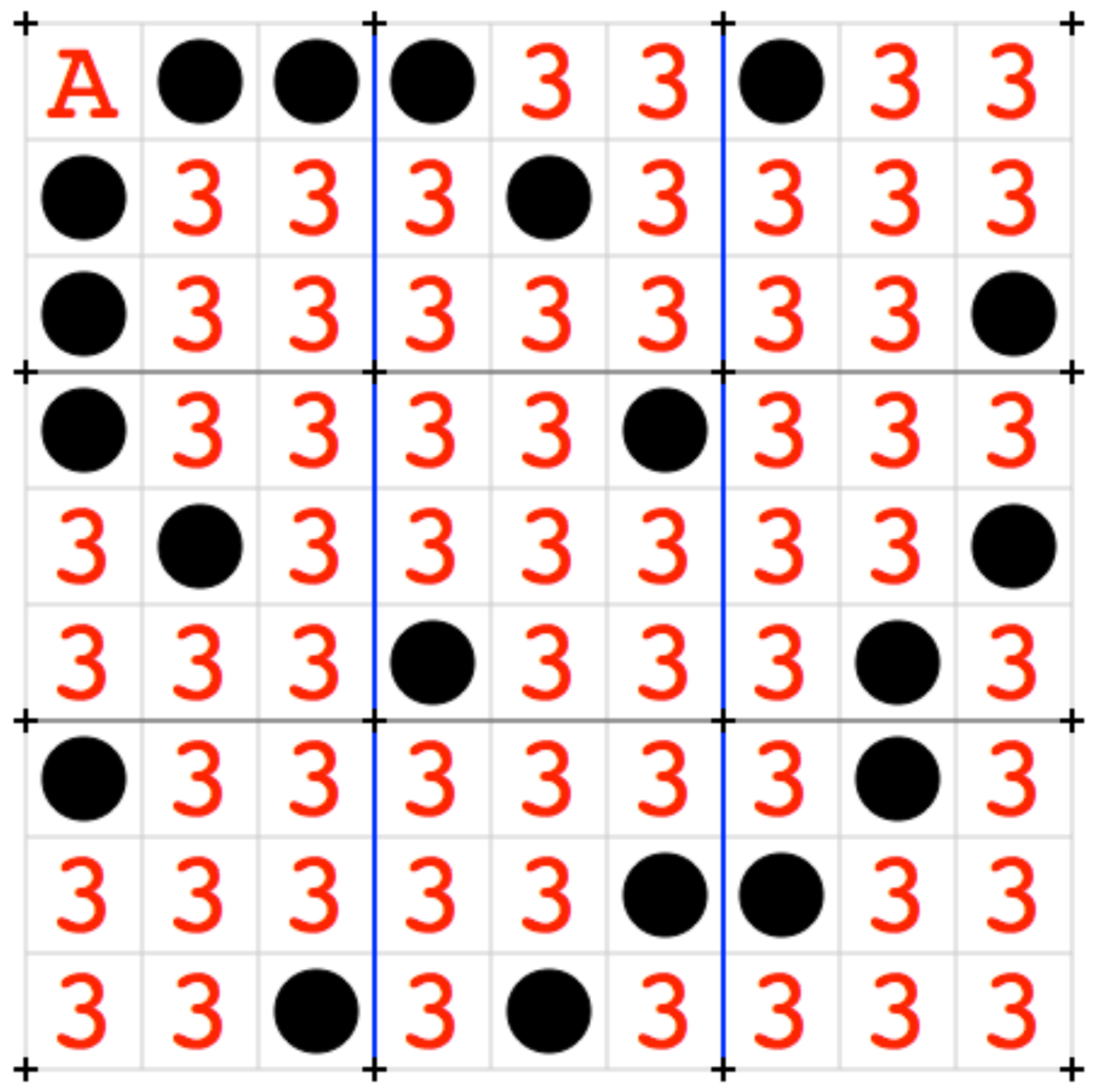}
\caption{The maximal cap $C$ from Figure \ref{F:AG43cap} showing how many lines each point not in $C$ completes with pairs of points in $C$.}
\label{F:Slines}
\end{center}
\end{figure}

We now define demicaps.  These substructures of a maximal cap  have the property that  the complement of a demicap in a maximal cap is also a demicap (see Proposition \ref{P:complement}); they will also have a close connection to the partitions of $AG(4,3)$ discussed in Section~\ref{S:cpsptns}. An example of a demicap is pictured in Figure \ref{F:DC}.

\begin{Def}
A \emph{demicap} $D$ is a cap in $AG(4,3)$ consisting of 5 pairs of points where the line through each pair  passes through a common point,  $a$, called the \emph{anchor point}, so that no four of these lines are co-hyperplanar.
\end{Def}

\begin{figure}[h]
\begin{picture}(100,80)\setlength{\unitlength}{.02cm}		%
\usebox{\grid}
\put (22.5,131){\circle* {8}}	
\put (37.5,131){\circle* {8}}	
\put (7.5,116){\circle* {8}}		
\put (7.5,101){\circle* {8}}		
\put (54.5,131){\circle* {8}}	
\put (101.5,131){\circle* {8}}	
\put (7.5,84){\circle* {8}}		
\put (131.5,69){\circle* {8}}	
\put (7.5,37){\circle* {8}}		
\put (69.5,7){\circle* {8}}		
\end{picture}\vspace{-.1in}
\begin{center}
\caption{A demicap with anchor point $\vec{0}$ in the upper left.}
\label{F:DC}
\end{center}
\end{figure}
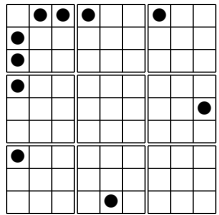

While the definition of a demicap assumes that no four pairs of points are co-hyperplanar, in fact, a demicap cannot have eight points in the same hyperplane even if you don't assume they are pairs of points that make a line with the anchor $a$.  Suppose we have five pairs of points so that each pair completes a line with $a$, and we know that eight of those points lie in the same hyperplane.  Since the eight points must include more than one of the original pairs, and hyperplanes are line-closed, $a$ must be in that hyperplane as well.  That means those eight points must be four of the pairs of points we started with.  (Note that a hyperplane is $AG(3,3)$, and the size of a maximal cap in that space is 9.) 

As the reader might notice from the above, the next definition will make our exposition neater.

\begin{Def}
Let $a$ be a point in $AG(4,3)$.  An \emph{$a$-line} is a pair of points $b,c$ in $AG(4,3)$ so that $\{a,b,c\}$ is a line in $AG(4,3)$.

\end{Def}

Using this definition, a demicap consists of five $a$-lines, where no four of them are co-hyperplanar.  
We will write a demicap as $D=\{d_1, d_1', d_2, d_2', \dots, d_5, d_5'\}$ so that  $d_i,d_i'$ is an $a$-line for $1\leq i\leq5$.

Proposition \ref{P:3lines} gives the number of lines that any point not in a maximal cap completes with pairs of points from that cap.  The next lemma shows that demicaps have a similar structure.

\begin{Lem}\label{L:Complete1Line}
Let $D$ be a cap consisting of five $a$-lines (where $a$ is necessarily not in $D$).  Then $D$ is a demicap with anchor point $a$ if and only if $D\cup \{b\}$ contains at most one line, for $b\not\in D\cup \{a\}$.
\end{Lem}
\begin{proof}

Let $a$ be a point and  $D=\{d_1, d_1', d_2, d_2', \dots, d_5, d_5'\}$ be a cap where $d_i,d_i'$ is an $a$-line for $1\leq i\leq5$.

First, suppose that there is a point $b\neq a$ so that $D\cup \{b\}$ contains two lines. We can relabel so that $\{b,d_1,d_2\}$ is a line.  This means that $a,b,d_1, d_1', d_2$, and  $d_2'$ are coplanar.  The second line in $D \cup b$ can't contain either $d_1'$ or $d_2'$, or the same plane would contain five points of $D$, but five points in a plane of $AG(4,3)$ must contain a line.  Relabel so that $\{b,d_3,d_4\}$ is a line.  This now implies that $a,b, d_3, d_3',d_4$ and $d_4'$ are coplanar.  Thus, $b,d_1, d_1', d_2, d_2', d_3, d_3',d_4$ and $d_4'$ are co-hyperplanar, so $D$ is not a demicap.

For the converse,  assume that $D$ is not a demicap. Because of the assumed structure of $D$, 
it must be the case that eight of the points of $D$, and thus also $a$, are in the same hyperplane $H$. The eight points can be paired in 28 ways; each pair has a point that completes a line with the pair.  Four of those pairs give rise to the anchor point, leaving  24 pairs.  However, there are only 27 points in $H$, and 8 of those points are in $D$.  It follows that there must be at least one point $b$ in $H$, excluding the anchor point, such that $\{b\}\cup D$ contains at least two lines. 
\end{proof}

The next corollary follows immediately from the previous proposition.

\begin{Cor}\label{C:demicapcompletecounts}
Let $D$ be a demicap. Then the remaining 71 points in $AG(4,3)$ consist of the anchor point $a$, 40 points completing exactly one line with two points from $D$, and 30 points completing no lines with points from $D$.
\end{Cor}

The next result  show that all demicaps are affinely equivalent. This is our first indication that the  notion could be a useful one.

\begin{Thm}\label{T:DCAE}
All demicaps are affinely equivalent.
\end{Thm}

\begin{proof}

Let $D$ be a demicap with anchor point $\vec0$, so that $D$ consists of five 
$\vec0$-lines. Let $d_1,d_1',d_2,d_2',d_3,d_3',d_4,d_4'$ be the points from  four  of those $\vec0$-lines in $D$. We will first show that there are exactly eight demicaps containing those points.   Now $\{\vec{d}_1,\vec{d}_2,\vec{d}_3,\vec{d}_4\}$ forms a basis for $AG(4,3)$ considered as the vector space $GF(3)^4$, as otherwise those four pairs would be co-hyperplanar.   Then if $d_5,d_5'$ are the last two points in $D$, it must be true that $\vec{d}_5 = \sum_{i=1}^4 c_i\vec{d}_i$ where $c_i = \pm1$. (If, for example, $c_4=0$, then $a,d_1,d_2,d_3,d_5$ are all in the same hyperplane, so there would be four pairs of co-hyperplanar points, violating the fact that $D$ is a demicap). There are exactly 16 choices for the coefficients $c_i$, giving rise to eight pairs of points that  $d_5, d_5'$ can be chosen from; those pairs are pictured in Figure \ref{F:DCcompleted}. Let $S$ be the set consisting of these 16 points.

\begin{figure}[h]
\begin{picture}(140,115)\setlength{\unitlength}{.028cm}		%
\multiput(0,94)(0,15){4}{\line(1,0){45}}		
\multiput(0,94)(15,0){4}{\line(0,1){45}}		
\multiput(47,94)(0,15){4}{\line(1,0){45}}	
\multiput(47,94)(15,0){4}{\line(0,1){45}}	
\multiput(94,94)(0,15){4}{\line(1,0){45}}	
\multiput(94,94)(15,0){4}{\line(0,1){45}}	
\multiput(0,47)(0,15){4}{\line(1,0){45}}		
\multiput(0,47)(15,0){4}{\line(0,1){45}}		
\multiput(47,47)(0,15){4}{\line(1,0){45}}	
\multiput(47,47)(15,0){4}{\line(0,1){45}}	
\multiput(94,47)(0,15){4}{\line(1,0){45}}	
\multiput(94,47)(15,0){4}{\line(0,1){45}}	
\multiput(0,0)(0,15){4}{\line(1,0){45}}		
\multiput(0,0)(15,0){4}{\line(0,1){45}}		
\multiput(47,0)(0,15){4}{\line(1,0){45}}	
\multiput(47,0)(15,0){4}{\line(0,1){45}}	
\multiput(94,0)(0,15){4}{\line(1,0){45}}	
\multiput(94,0)(15,0){4}{\line(0,1){45}}	
%
{\footnotesize 
\put (5,129){$a$}		
\put (18,129){$d_1$}	
\put (32,129){$d_1'$}	
\put (3,114){$d_2$}		
\put (3,99){$d_2'$}		
\put (50,129){$d_3$}	
\put (97,129){$d_3'$}	
\put (3,82){$d_4$}		
\put (3,35){$d_4'$}		
\put (67.5,67){$s $}		
\put (82.5,67){$t $}		
\put (67.5,52){$u $}		
\put (82.5,52){$v $}		
\put (113.5,67){$w $}	
\put (129.5,67){$x $}	
\put (114.5,52){$y $}
\put (129.5,52){$z $}	
\put (65,20){$z' $}			
\put (81,20){$y' $}		
\put (65,5){$x' $}		
\put (80,5){$w' $}		
\put (112,20){$v' $}	
\put (127,20){$u' $}	
\put (112,5){$t' $}		
\put (127,5){$s' $}	}	
\end{picture}\vspace{-.1in}
\begin{center}
\caption{Four non-co-hyperplanar pairs of points $d_1,d_1', d_2,d_2' ,d_3,d_3',d_4,d_4'$ making four lines through an anchor point $a$; any of the labeled eight pairs of points could be added to make a demicap.}
\label{F:DCcompleted}
\end{center}
\end{figure}

We will now show the eight demicaps found above are affinely equivalent. Consider the linear transformations of $AG(4,3)$ that map each of the basis elements  $\{\vec{d}_1,\vec{d}_2,\vec{d}_3,\vec{d}_4\}$ to either itself or to its negative; there are 16 such transformations. If $\vec{s}=\sum_{i=1}^4 \vec{d}_i$, then each of the points  in $S$ 
is the image of $s$ under one of those transformations.  Thus, all eight of the possible demicaps containing $d_1,d_1', d_2,d_2' ,d_3,d_3',d_4,d_4'$ are affinely equivalent.

Finally, any other demicap $D'$ with anchor point $a$ consists of five $a$-lines $f_i,f_i', 1\leq i\leq5$.  There is a unique affine transformation mapping $a'$ to $\vec0$ and $f_i$ to $d_i$, $1\leq i\leq5$.
That transformation then maps $\{f_5,f_5'\}$ to one of the pairs of points in $S$, so $D'$ is affinely equivalent to a demicap which is affinely equivalent to $D$.  Thus, all demicaps are affinely equivalent.
\end{proof}

The next corollary follows from the proof of Theorem \ref{T:DCAE}.

\begin{Cor}
Let $a$ be a point in $AG(4,3)$.  Given a set of four non-co-hyperplanar $a$-lines, there are exactly eight demicaps that contain them.
\end{Cor}

The next proposition counts the number of demicaps with a given anchor. We also count how many demicaps a given maximal cap contains and how many maximal caps a given demicap is in. These counts will provide a link to the maximal caps in a 1-completable pair with $C$ in Section \ref{S:results}.

\begin{Prop}\label{P:dcinfo} In $AG(4,3)$:

\begin{enumerate}
\item There are  101,088 distinct demicaps with a given anchor point $a$. 
\item Let $C$ be a maximal cap in $AG(4,3)$.  Then $C$ contains exactly 72 demicaps.
\item Every demicap is contained in exactly 6 maximal caps.
\end{enumerate}
\end{Prop}

\begin{proof}
(1) 
We count the number of ways to choose five $a$-lines so that no four are co-hyperplanar.
There are 40 lines through $a$, so we can choose the first two $a$-lines to be the points from any two of those lines.  Those four  points determine a plane which contains two other $a$-lines, and we cannot add any of those $a$-lines to our demicap. This means that there are 36 $a$-lines we may choose for the third $a$-line. 
The three $a$-lines chosen to this point determine a hyperplane which contains 13 $a$-lines, so we cannot choose any of those for our fourth $a$-line.  Thus, we may choose any of the 27 remaining $a$-lines for the fourth $a$-line in the demicap. The argument from the proof of Theorem \ref{T:DCAE} shows that we could choose any of eight $a$-lines  to complete the demicap. Since the order in which we chose the $a$-lines  doesn't matter, this gives $\frac{40\cdot 39\cdot36\cdot27\cdot8}{5!}=
101,088$ demicaps.

(2)
Let $C$ be an arbitrary maximal cap (so $C$ consists of 10 $a$-lines). To count the number of demicaps in  $C$, we count the number of ways to choose five $a$-lines from the 10 so that no four $a$-lines are co-hyperplanar.  We may choose any three of the 10 pairs at the start. 

When we partition $AG(4, 3)$ into three disjoint and parallel hyperplanes $H_1$, $H_2$ and $H_3$, the intersection of those hyperplanes with $C$ partitions the cap. The only possible distribution of points in that partition is $\{9,9,2\}$ or $\{8,6,6\}$. (This result is well-known.  Davis and Maclagan \cite{MR2005098} call it a hyperplane triple, and Potechin \cite{MR2372838}  calls it a division.)  If we assume that the anchor point $a$ is in $H_1$, then  the number of points from $C$ in $H_2$ and $H_3$ must be equal, since any line in $AG(4,3)$ is either contained in one hyperplane or it has one point from each of three parallel hyperplanes. It follows that $H_1$ either contains two or eight points of $C$.

Now, consider the partition of $AG(4,3)$ into three parallel hyperplanes $H_1'$, $H_2'$, $H_3'$, where $H_1'$ is determined by the three pairs of points initially chosen (which forces $H_2'$ and $H_3'$).  There must be a fourth pair of points from $C$ in $H_1'$, by the previous paragraph, and we may not choose that pair to be in the demicap.  Thus, we may freely choose one of the six remaining pairs not in this hyperplane.  

How many choices do we have for the last pair of points from $C$? If we consider the hyperplane containing two of the three initially chosen pairs of points plus the fourth chosen pair, that hyperplane will contain exactly one more pair of points (since it must contain eight points of $C$), which we cannot add to our demicap.  This excludes three of the five remaining pairs.  Either of the two remaining pairs can indeed make a demicap with the first four chosen.  Since the order we chose the pairs doesn't matter, the total number of demicaps in a given maximal cap is $\frac{10\cdot9\cdot8\cdot6\cdot2}{5!}=72$ demicaps.  

(3) 
There are 8424 maximal caps with anchor point $a$ (a proof of this can be found in \cite{MR3262358}); from (1), there are 101,088 demicaps with the same anchor point.  From (2), each maximal cap contains 72 demicaps.  An incidence count then shows that each demicap is contained in 6 maximal caps.
\end{proof}

An alternate proof of (1) above uses the Orbit-Stabilizer Theorem.  If we assume the anchor point is $\vec0$, then all affine transformations fixing $\vec0$ are actually linear transformations. Thus, the full group acting on the set of demicaps with anchor point $\vec0$ is $GL(4,3)$, of size 24,261,120.  It's an illuminating exercise to show that the stabilizer of a given demicap is $S_5\times \mathbb{Z}_2$, of size 240.  

\medskip
Our next result shows that the 72 demicaps come in 36 pairs whose union is $C$. In the next section, this will provide a direct connection between demicaps in $C$ and the maximal caps $C'$ in a 1-completable pair with $C$.

\begin{Prop}\label{P:complement}
Let $C$ be a maximal cap in $AG(4,3)$, and let $D$ be a demicap with $D\subset C$.  Then $C \backslash D$ is also a demicap.  Thus, every maximal cap can be written as the disjoint union of two demicaps in 36 ways.
\end{Prop}

\begin{proof}
We will show that the complement in $C$ of four co-hyperplanar $a$-lines cannot contain a demicap. Once we have done so, then if  $D$ is a demicap contained in $C$, then its complement $C\backslash D$ cannot contain four co-hyperplanar $a$-lines.  Since the complement must consist of five $a$-lines, it must necessarily be a demicap.    

The reader can follow what we do here by considering the maximal cap in Figure~\ref{F:AG43cap} and the four $a$-lines in the top hyperplane, which we call $H_1$.
These eight points and the anchor point are nine of the 27 points in  $H_1$.  Consider $H_2$ and $H_3$, the two hyperplanes parallel to $H_1$.  The complement in $C$ of those eight points consists of six  $a$-lines; thus there must be six points in $H_2$, and six in $H_3$, as hyperplanes are line-closed and $a \in H_1$.  If there were a demicap contained in those 12 points, it would have to consist of five points in $H_2$, and five points in $H_3$. Consider all pairings of these 10  points with one from $H_2$ and one from $H_3$.  There are $5\cdot5 = 25$ such pairings, and each gives rise to a point in $H_1$.  Five of those are the anchor point, so the remaining 20 must be points other than $a$ and the eight original points, which leaves 18 points.  Thus, there must be a point in $H_1$ that completes two lines from the potential demicap, violating Lemma \ref{L:Complete1Line}.  Thus, the 12 points that are the complement of four $a$-lines which are co-hyperplanar cannot contain a demicap, as desired.

The last sentence of  the proposition follows from Proposition \ref{P:dcinfo} (2).
\end{proof}

We are now ready to make the connection between demicaps and pairs of maximal caps $\{C,C'\}$ that are in exactly one partition.


\section{Demicaps of $C$ and maximal caps that are in one partition with $C$}\label{S:results}

In $AG(4,3)$, any pair of disjoint maximal caps is in either one, two or six partitions into four maximal caps along with their common anchor point. Given a partition into four maximal caps plus their common anchor point, either there are two 1-completable pairs or there are two 2-completable pairs; in both cases, all four of the other pairs are 6-completable pairs.  All this can be found in \cite{MR3262358}.  Our goal in this section is to use one pair of demicaps whose union is $C$ to find one of the 36 maximal caps that makes a 1-completable pair with $C$, and then to find the unique partition of $AG(4,3)$ containing that pair. 

\newsavebox{\gridsm}
\savebox{\gridsm}(0,59){ \setlength{\unitlength}{.015cm}
\multiput(0,94)(0,15){4}{\line(1,0){45}}		
\multiput(0,94)(15,0){4}{\line(0,1){45}}		
\multiput(47,94)(0,15){4}{\line(1,0){45}}	
\multiput(47,94)(15,0){4}{\line(0,1){45}}	
\multiput(94,94)(0,15){4}{\line(1,0){45}}	
\multiput(94,94)(15,0){4}{\line(0,1){45}}	
\multiput(0,47)(0,15){4}{\line(1,0){45}}		
\multiput(0,47)(15,0){4}{\line(0,1){45}}		
\multiput(47,47)(0,15){4}{\line(1,0){45}}	
\multiput(47,47)(15,0){4}{\line(0,1){45}}	
\multiput(94,47)(0,15){4}{\line(1,0){45}}	
\multiput(94,47)(15,0){4}{\line(0,1){45}}	
\multiput(0,0)(0,15){4}{\line(1,0){45}}		
\multiput(0,0)(15,0){4}{\line(0,1){45}}		
\multiput(47,0)(0,15){4}{\line(1,0){45}}	
\multiput(47,0)(15,0){4}{\line(0,1){45}}	
\multiput(94,0)(0,15){4}{\line(1,0){45}}	
\multiput(94,0)(15,0){4}{\line(0,1){45}}	
}


\subsection{One pair of demicaps whose union is $C$ determines a maximal cap that is in exactly one partition with $C$}

Every  maximal cap $C$ contains 72 demicaps. Since the complement of a demicap in $C$ is also a demicap, there are  36 pairs of disjoint demicaps whose union is $C$.  Those pairs provide a direct link to the 36 maximal caps disjoint from $C$ which make a 1-completable pair with $C$.  

\medskip
Let $C$ be a maximal cap, and $D$ a demicap contained in $C$.  Let $D'$ be the complement of $D$ in $C$.  Figure \ref{F:Cunion} shows such a union for the maximal cap pictured in Figure~\ref{F:AG43cap} and the demicap pictured in Figure \ref{F:DC}.

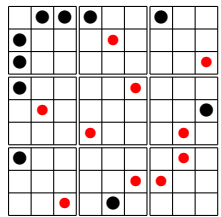
\begin{figure}[h]
\begin{picture}(94,80)\setlength{\unitlength}{.02cm}		%
\usebox{\grid}
\put (22.5,131){\circle* {8}}	
\put (37.5,131){\circle* {8}}	
\put (7.5,116){\circle* {8}}		
\put (7.5,101){\circle* {8}}		
\put (54.5,131){\circle* {8}}	
\put (101.5,131){\circle* {8}}	
\put (7.5,84){\circle* {8}}		
\put (131.5,69){\circle* {8}}	
\put (7.5,37){\circle* {8}}		
\put (69.5,7){\circle* {8}}		
\put (69.5,116){\textcolor{red}{\circle* {7}}}	
\put (131.5,101){\textcolor{red}{\circle* {7}}}	
\put (22.5,69){\textcolor{red}{\circle* {7}}}	
\put (84.5,84){\textcolor{red}{\circle* {7}}}	
\put (54.5,54){\textcolor{red}{\circle* {7}}}	
\put (116.5,54){\textcolor{red}{\circle* {7}}}	
\put (37.5,7){\textcolor{red}{\circle* {7}}}		
\put (84.5,22){\textcolor{red}{\circle* {7}}}	
\put (116.5,37){\textcolor{red}{\circle* {7}}}	
\put (101.5,22){\textcolor{red}{\circle* {7}}}	
\end{picture}
\begin{center}
\caption{The maximal cap $C$ from Figure~\ref{F:AG43cap} as the union of two demicaps $D$ (from Figure \ref{F:DC}) and $D'$.}
\label{F:Cunion}
\end{center}
\end{figure}

The two demicaps $D$ and $D'$ each have 40 points that complete one line with a pair of points in the demicap, by Corollary \ref{C:demicapcompletecounts}. Output from the Cap Builder \cite{Capbuilder}, shown in Figure \ref{F:DD'1s}, indicates the points in $AG(4,3)$ that complete one line with two points in $D$ on the left and $D'$ in the middle.  The anchor point $a$ is the point that completes 5 lines in each demicap.

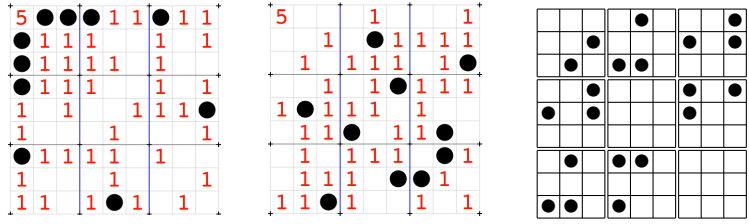
\begin{figure}[htbp]
\includegraphics[width=1.12in]{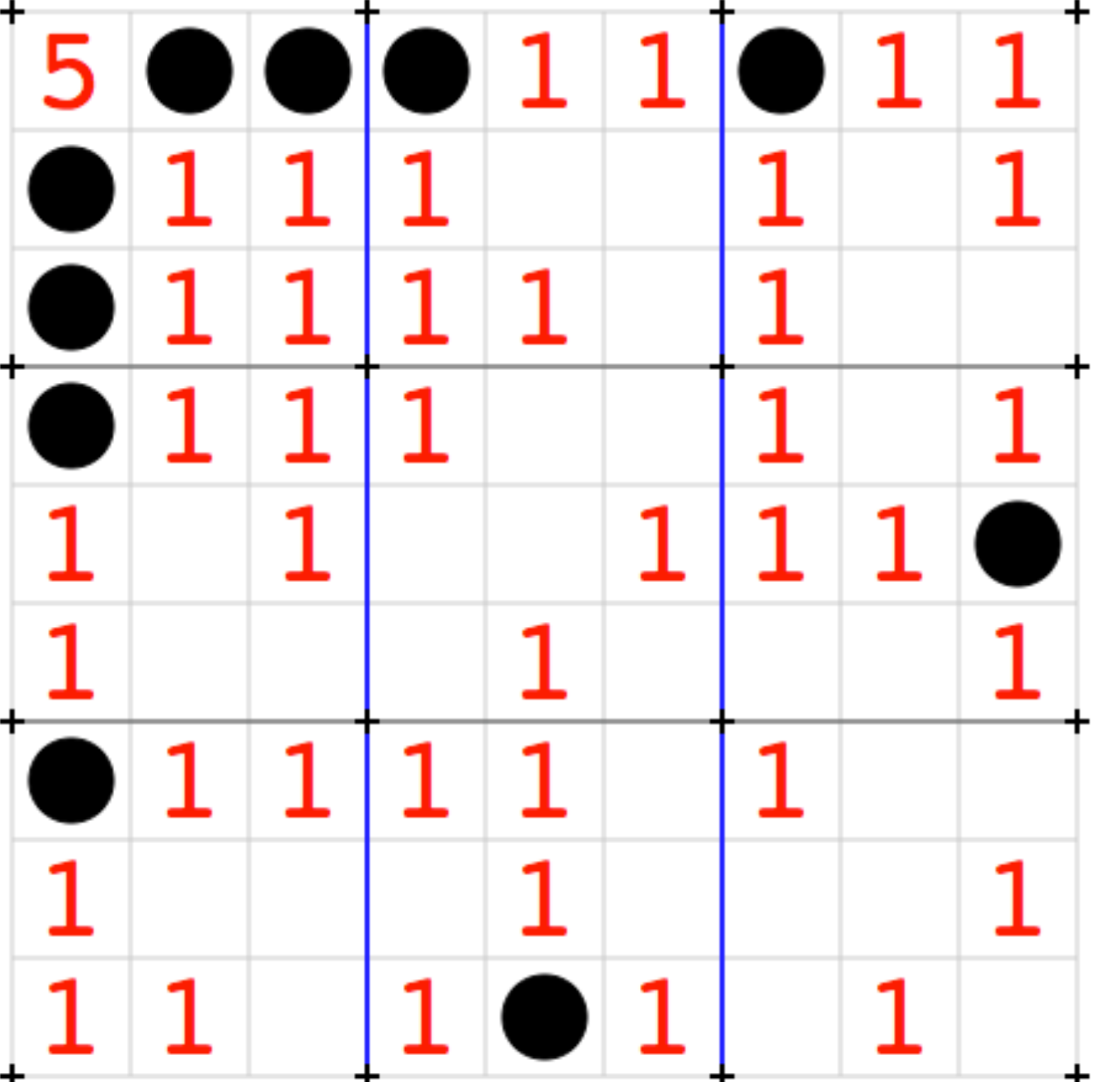} \hspace{.15in}
\includegraphics[width=1.12in]{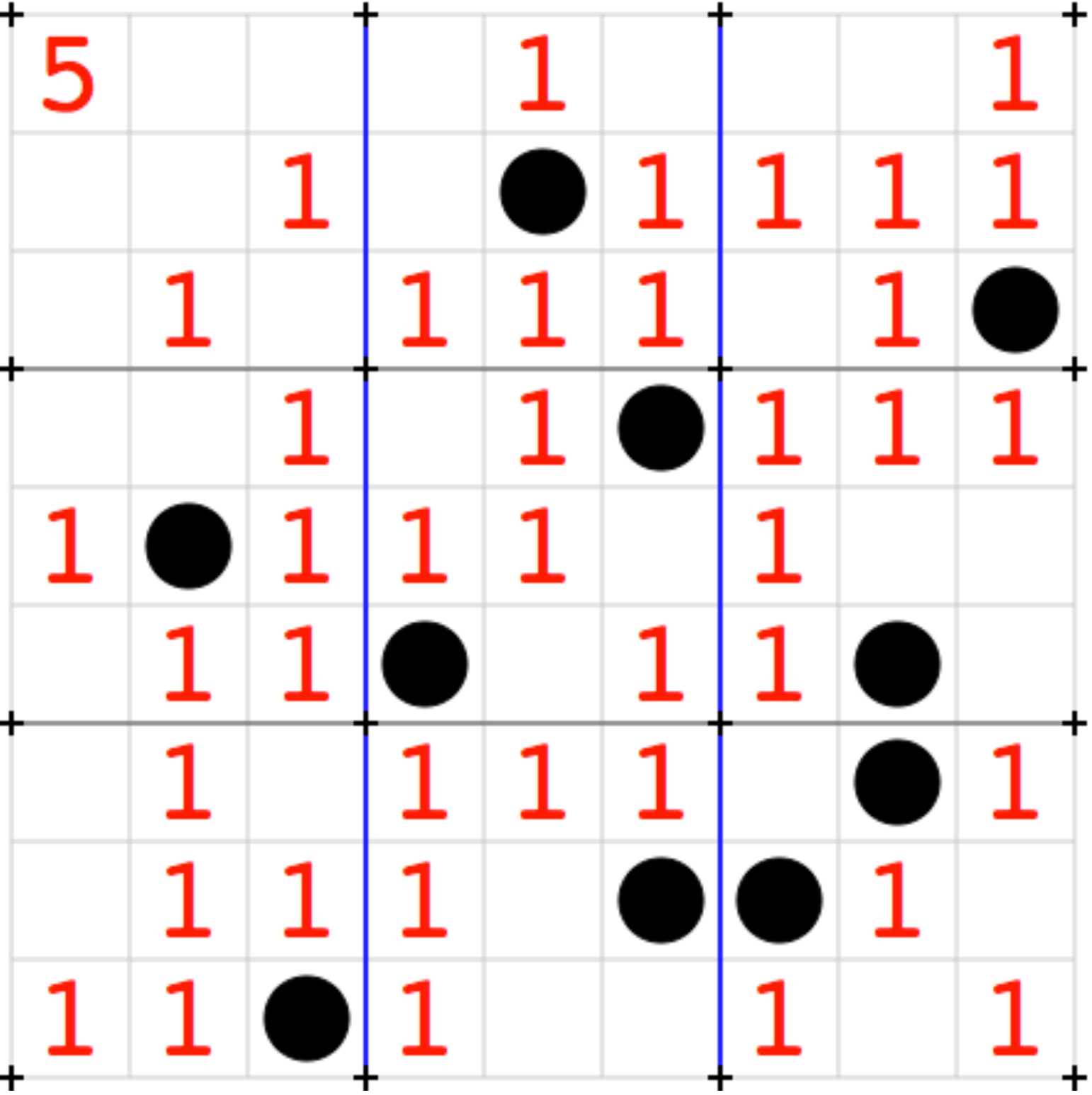} \hspace{.15in}
\begin{picture}(96,84)\setlength{\unitlength}{.02cm}		%
\usebox{\grid}
\put (37.5,116){\circle* {8}}	
\put (22.5,101){\circle* {8}}	
\put (69.5,131){\circle* {8}}	
\put (54.5,101){\circle* {8}}	
\put (69.5,101){\circle* {8}}	
\put (131.5,131){\circle* {8}}	
\put (101.5,116){\circle* {8}}	
\put (131.5,116){\circle* {8}}	
\put (37.5,84){\circle* {8}}		
\put (7.5,69){\circle* {8}}		
\put (37.5,69){\circle* {8}}		
\put (101.5,84){\circle* {8}}	
\put (131.5,84){\circle* {8}}	
\put (101.5,69){\circle* {8}}	
\put (22.5,37){\circle* {8}}		
\put (7.5,7){\circle* {8}}		
\put (22.5,7){\circle* {8}}		
\put (54.5,37){\circle* {8}}		
\put (69.5,37){\circle* {8}}		
\put (54.5,7){\circle* {8}}		
\end{picture}
\caption{The points in $AG(4,3)$ that complete one line with points from demicaps $D$  (left) and $D'$ (center) are indicated by 1's. The points that complete one line with a pair of points from both $D$  and $D'$ are shown on the right; those points are a maximal cap $C'$. $C$ and $C'$ are a 1-completable pair.}
\label{F:DD'1s}
\end{figure}

Comparing the results on the left and in the center of Figure \ref{F:DD'1s}, one can see that there are exactly 20 points  that complete a line with a pair of points from $D$ and also complete a line with a pair of points from $D'$ (the anchor point completes five lines from each).  Those 20 points are shown on the right in Figure \ref{F:DD'1s}. The 20 points shown on the right in Figure \ref{F:DD'1s} do indeed form a maximal cap, as the reader can verify.  Furthermore, the authors have verified that the original maximal cap $C$ and this new maximal cap $C'$ are in exactly one partition.

Finally, the authors have also verified that each of the 36 pairs of disjoint demicaps gives rise to a different maximal cap using this procedure, so the 36 pairs of disjoint demicaps whose union is $C$ correspond uniquely to the 36 maximal caps that make a 1-completable pair with $C$.

We summarize these results in a theorem. 

\begin{Thm}\label{T:demicapcorresp}
Let $C$ be a maximal cap. There are 36 pairs of disjoint demicaps whose union is $C$ and there are 36 maximal caps that make a 1-completable pair with $C$. A pair of demicaps $D$ and $D'$ where $C=D\cup D'$ corresponds uniquely to a maximal cap $C'$ that is in exactly one partition with $C$: $C'$ consists of all the points in $AG(4,3)$ that complete one line with a pair of points in $D$ and  one line with a pair of points in $D'$. 
\end{Thm}

Because all maximal caps are affinely equivalent and any affine transformation sends demicaps to demicaps, verifying the theorem for one example of a maximal cap suffices as a proof of the theorem.

\subsection{The unique partition containing $C$ and $C'$.}

We have found a bijection between the complementary pairs of demicaps whose union is $C$ and the  maximal caps which make a 1-completable pair with $C$. Is it possible to find the unique partition containing $C$ and $C'$?  It turns out we already have almost everything we need to accomplish this.  

Recall, we started with a demicap decomposition $C=D\cup D'$; we then found the 40 points completing a line with points from $D$ and the 40 completing a line with points from $D'$, as pictured in Figure \ref{F:DD'1s}. Let $S_1$ be the 20 points that complete a line with points of $D$ that do not complete a line with points from $D'$.  $S_1$ contains lines, so it is not a maximal cap, however, $S_1$ contains 12 demicaps, in six complementary pairs.
 A particular one of  those decompositions is shown in Figure~\ref{F:Find4demis1}, with the two demicaps denoted $D_{1}$ and $D_{2}$.  The same is true of $S_2$, the 20 points that complete one line with two points from $D'$ but not from $D$; an example of those two decompositions is shown in Figure~\ref{F:Find4demis2}, where the demicaps we find are denoted $D_{1}'$ and $D_{2}'$.

\begin{figure}[h]
\begin{picture}(70,70)\setlength{\unitlength}{.015cm}		%
\usebox{\gridsm}
\put (22.5,116){\textcolor{red}{\circle* {8}}}	
\put (37.5,101){\textcolor{red}{\circle* {8}}}
\put (84.5,131){\textcolor{Tan}{\circle* {8}}}	
\put (54.5,116){\textcolor{red}{\circle* {8}}}
\put (116.5,131){\textcolor{Tan}{\circle* {8}}}	
\put (101.5,101){\textcolor{red}{\circle* {8}}}	
\put (22.5,84){\textcolor{red}{\circle* {8}}}		
\put (7.5,54){\textcolor{Tan}{\circle* {8}}}		
\put (54.5,84){\textcolor{Tan}{\circle* {8}}}		
\put (84.5,69){\textcolor{red}{\circle* {8}}}	
\put (69.5,54){\textcolor{red}{\circle* {8}}}	
\put (116.5,69){\textcolor{Tan}{\circle* {8}}}	
\put (131.5,54){\textcolor{Tan}{\circle* {8}}}	
\put (37.5,37){\textcolor{red}{\circle* {8}}}		
\put (7.5,22){\textcolor{Tan}{\circle* {8}}}		
\put (69.5,22){\textcolor{Tan}{\circle* {8}}}		
\put (84.5,7){\textcolor{Tan}{\circle* {8}}}		
\put (101.5,37){\textcolor{Tan}{\circle* {8}}}	
\put (131.5,22){\textcolor{red}{\circle* {8}}}
\put (116.5,7){\textcolor{red}{\circle* {8}}}	
\end{picture}
\begin{picture}(70,70)\setlength{\unitlength}{.015cm}		%
\usebox{\gridsm}
\put (22.5,116){\textcolor{red}{\circle* {8}}}	
\put (37.5,101){\textcolor{red}{\circle* {8}}}
\put (54.5,116){\textcolor{red}{\circle* {8}}}
\put (101.5,101){\textcolor{red}{\circle* {8}}}	
\put (22.5,84){\textcolor{red}{\circle* {8}}}	
\put (84.5,69){\textcolor{red}{\circle* {8}}}	
\put (69.5,54){\textcolor{red}{\circle* {8}}}		
\put (37.5,37){\textcolor{red}{\circle* {8}}}		
\put (131.5,22){\textcolor{red}{\circle* {8}}}
\put (116.5,7){\textcolor{red}{\circle* {8}}}	
\end{picture}
\begin{picture}(70,70)\setlength{\unitlength}{.015cm}		%
\usebox{\gridsm}
\put (84.5,131){\textcolor{Tan}{\circle* {8}}}	
\put (116.5,131){\textcolor{Tan}{\circle* {8}}}	
\put (7.5,54){\textcolor{Tan}{\circle* {8}}}		
\put (54.5,84){\textcolor{Tan}{\circle* {8}}}		
\put (116.5,69){\textcolor{Tan}{\circle* {8}}}	
\put (131.5,54){\textcolor{Tan}{\circle* {8}}}	
\put (7.5,22){\textcolor{Tan}{\circle* {8}}}		
\put (69.5,22){\textcolor{Tan}{\circle* {8}}}		
\put (84.5,7){\textcolor{Tan}{\circle* {8}}}		
\put (101.5,37){\textcolor{Tan}{\circle* {8}}}	
\end{picture}
\caption{Left: $S_1$ (20 points that complete a line with a pair of points from $D$ but not $D'$). Middle: $D_{1}$. Right: $D_{2}$ } 
\label{F:Find4demis1}

\begin{picture}(70,70)\setlength{\unitlength}{.015cm}		%
\usebox{\gridsm}
\put (84.5,116){\textcolor{blue}{\circle* {8}}}
\put (84.5,101){\textcolor{blue}{\circle* {8}}}
\put (116.5,116){\textcolor{blue}{\circle* {8}}}	
\put (116.5,101){\textcolor{blue}{\circle* {8}}}	
\put (22.5,54){\textcolor{blue}{\circle* {8}}}	
\put (37.5,54){\textcolor{blue}{\circle* {8}}}	
\put (69.5,84){\textcolor{green}{\circle* {8}}}		
\put (54.5,69){\textcolor{green}{\circle* {8}}}		
\put (69.5,69){\textcolor{green}{\circle* {8}}}		
\put (84.5,54){\textcolor{blue}{\circle* {8}}}	
\put (116.5,84){\textcolor{green}{\circle* {8}}}	
\put (101.5,54){\textcolor{green}{\circle* {8}}}	
\put (22.5,22){\textcolor{blue}{\circle* {8}}}		
\put (37.5,22){\textcolor{blue}{\circle* {8}}}	
\put (84.5,37){\textcolor{green}{\circle* {8}}}		
\put (54.5,22){\textcolor{green}{\circle* {8}}}		
\put (131.5,37){\textcolor{green}{\circle* {8}}}	
\put (116.5,22){\textcolor{blue}{\circle* {8}}}
\put (101.5,7){\textcolor{green}{\circle* {8}}}		
\put (131.5,7){\textcolor{green}{\circle* {8}}}		
\end{picture}
\begin{picture}(70,70)\setlength{\unitlength}{.015cm}		%
\usebox{\gridsm}
\put (84.5,116){\textcolor{blue}{\circle* {8}}}
\put (84.5,101){\textcolor{blue}{\circle* {8}}}
\put (116.5,116){\textcolor{blue}{\circle* {8}}}	
\put (116.5,101){\textcolor{blue}{\circle* {8}}}	
\put (22.5,54){\textcolor{blue}{\circle* {8}}}	
\put (37.5,54){\textcolor{blue}{\circle* {8}}}	
\put (84.5,54){\textcolor{blue}{\circle* {8}}}	
\put (22.5,22){\textcolor{blue}{\circle* {8}}}		
\put (37.5,22){\textcolor{blue}{\circle* {8}}}	
\put (116.5,22){\textcolor{blue}{\circle* {8}}}
\end{picture}
\begin{picture}(70,70)\setlength{\unitlength}{.015cm}		%
\usebox{\gridsm}
\put (69.5,84){\textcolor{green}{\circle* {8}}}	
\put (54.5,69){\textcolor{green}{\circle* {8}}}		
\put (69.5,69){\textcolor{green}{\circle* {8}}}		
\put (116.5,84){\textcolor{green}{\circle* {8}}}	
\put (101.5,54){\textcolor{green}{\circle* {8}}}	
\put (84.5,37){\textcolor{green}{\circle* {8}}}		
\put (54.5,22){\textcolor{green}{\circle* {8}}}		
\put (131.5,37){\textcolor{green}{\circle* {8}}}	
\put (101.5,7){\textcolor{green}{\circle* {8}}}		
\put (131.5,7){\textcolor{green}{\circle* {8}}}		
\end{picture}
\caption{Left: $S_2$ (20 points that complete a line with a pair of points from $D'$ but not $D$). Middle:  $D_{1}'$. Right:  $D_{2}'$. } 
\label{F:Find4demis2}
\end{figure}

\begin{figure}[h]
\begin{picture}(70,70)\setlength{\unitlength}{.015cm}		%
\usebox{\gridsm}
\put (22.5,116){\textcolor{red}{\circle* {8}}}	
\put (37.5,101){\textcolor{red}{\circle* {8}}}	
\put (54.5,116){\textcolor{red}{\circle* {8}}}	
\put (101.5,101){\textcolor{red}{\circle* {8}}}		
\put (22.5,84){\textcolor{red}{\circle* {8}}}		
\put (84.5,69){\textcolor{red}{\circle* {8}}}		
\put (69.5,54){\textcolor{red}{\circle* {8}}}		
\put (37.5,37){\textcolor{red}{\circle* {8}}}		
\put (131.5,22){\textcolor{red}{\circle* {8}}}		
\put (116.5,7){\textcolor{red}{\circle* {8}}}			
\put (84.5,116){\textcolor{blue}{\circle* {8}}}	
\put (84.5,101){\textcolor{blue}{\circle* {8}}}	
\put (116.5,116){\textcolor{blue}{\circle* {8}}}	
\put (116.5,101){\textcolor{blue}{\circle* {8}}}	
\put (22.5,54){\textcolor{blue}{\circle* {8}}}		
\put (37.5,54){\textcolor{blue}{\circle* {8}}}		
\put (84.5,54){\textcolor{blue}{\circle* {8}}}		
\put (22.5,22){\textcolor{blue}{\circle* {8}}}		
\put (37.5,22){\textcolor{blue}{\circle* {8}}}		
\put (116.5,22){\textcolor{blue}{\circle* {8}}}	
\end{picture}
\begin{picture}(70,70)\setlength{\unitlength}{.015cm}		%
\usebox{\gridsm}
\put (84.5,131){\textcolor{Tan}{\circle* {8}}}	
\put (116.5,131){\textcolor{Tan}{\circle* {8}}}	
\put (7.5,54){\textcolor{Tan}{\circle* {8}}}		
\put (54.5,84){\textcolor{Tan}{\circle* {8}}}		
\put (116.5,69){\textcolor{Tan}{\circle* {8}}}	
\put (131.5,54){\textcolor{Tan}{\circle* {8}}}	
\put (7.5,22){\textcolor{Tan}{\circle* {8}}}		
\put (69.5,22){\textcolor{Tan}{\circle* {8}}}		
\put (84.5,7){\textcolor{Tan}{\circle* {8}}}		
\put (101.5,37){\textcolor{Tan}{\circle* {8}}}	
\put (69.5,84){\textcolor{green}{\circle* {8}}}		
\put (54.5,69){\textcolor{green}{\circle* {8}}}		
\put (69.5,69){\textcolor{green}{\circle* {8}}}		
\put (116.5,84){\textcolor{green}{\circle* {8}}}	
\put (101.5,54){\textcolor{green}{\circle* {8}}}	
\put (84.5,37){\textcolor{green}{\circle* {8}}}		
\put (54.5,22){\textcolor{green}{\circle* {8}}}		
\put (131.5,37){\textcolor{green}{\circle* {8}}}	
\put (101.5,7){\textcolor{green}{\circle* {8}}}		
\put (131.5,7){\textcolor{green}{\circle* {8}}}		
\end{picture}
\caption{Left: $M_1=D_{1}\cup D_{1}'$. Right:  $M_2=D_{2}\cup D_{2}'$. These are the other two maximal caps in the partition containing $C$ and $C'$. } 
\label{F:6comps}
\end{figure}
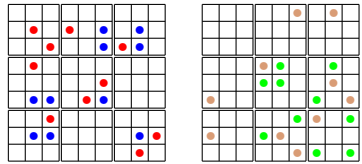

Taking all possible pairings of a demicap in $S_1$ with a demicap in $S_2$, there are only two pairings that form  a maximal cap. (For all other pairings, there are points in one demicap that complete a line with two points from the other demicap.) Further, the two demicaps in $S_1$ are a complementary pair, as are the two demicaps they are matched with in $S_2$.  The results obtained by pairing the demicaps from Figures~\ref{F:Find4demis1} and \ref{F:Find4demis2} are shown in Figure~\ref{F:6comps}.  We will call these maximal caps $M_1$ and $M_2$.   The reader can verify that $M_1$ and $M_2$ are both maximal caps. 

The two maximal caps found in this way are the two other maximal caps in the unique partition containing $C$ and $C'$; they are themselves a 1-completable pair.  We now have the unique partition that contains $C$ and $C'$:  $\{{\{a\}},C,C',M_1,M_2\}$. That partition is pictured in Figure \ref{F:Partition}; notice that it is a different partition than the one in Figure \ref{F:AG43ptn}. 

The demicaps whose unions produced $M_1$ and $M_2$ are special, not surprisingly.  The pair of demicaps in $M_1$ is exactly the pair that corresponds to $M_2$, via the correspondence given in Theorem~\ref{T:demicapcorresp}.  Similarly, the pair of demicaps in $M_2$ corresponds to $M_1$.

\begin{figure}[h]
\begin{picture}(70,70)\setlength{\unitlength}{.02cm}		%
\usebox{\grid}
\put (3,122){*}	  
\put (22.5,131){\circle* {8}}	
\put (37.5,131){\circle* {8}}	
\put (7.5,116){\circle* {8}}		
\put (7.5,101){\circle* {8}}		
\put (54.5,131){\circle* {8}}	
\put (69.5,116){\circle* {8}}	
\put (101.5,131){\circle* {8}}	
\put (131.5,101){\circle* {8}}	
\put (7.5,84){\circle* {8}}		
\put (22.5,69){\circle* {8}}		
\put (84.5,84){\circle* {8}}		
\put (54.5,54){\circle* {8}}		
\put (131.5,69){\circle* {8}}	
\put (116.5,54){\circle* {8}}	
\put (7.5,37){\circle* {8}}		
\put (37.5,7){\circle* {8}}		
\put (84.5,22){\circle* {8}}		
\put (69.5,7){\circle* {8}}		
\put (116.5,37){\circle* {8}}	
\put (101.5,22){\circle* {8}}	
\put (37.5,116){\textcolor{cyan}{\circle* {6}}	}
\put (22.5,101){\textcolor{cyan}{\circle* {6}}}	
\put (69.5,131){\textcolor{cyan}{\circle* {6}}}	
\put (54.5,101){\textcolor{cyan}{\circle* {6}}}	
\put (69.5,101){\textcolor{cyan}{\circle* {6}}}	
\put (131.5,131){\textcolor{cyan}{\circle* {6}}}	
\put (101.5,116){\textcolor{cyan}{\circle* {6}}}	
\put (131.5,116){\textcolor{cyan}{\circle* {6}}}	
\put (37.5,84){\textcolor{cyan}{\circle* {6}}}		
\put (7.5,69){\textcolor{cyan}{\circle* {6}}}		
\put (37.5,69){\textcolor{cyan}{\circle* {6}}}		
\put (101.5,84){\textcolor{cyan}{\circle* {6}}	}
\put (131.5,84){\textcolor{cyan}{\circle* {6}}	}
\put (101.5,69){\textcolor{cyan}{\circle* {6}}}	
\put (22.5,37){\textcolor{cyan}{\circle* {6}}}		
\put (7.5,7){\textcolor{cyan}{\circle* {6}}}		
\put (22.5,7){\textcolor{cyan}{\circle* {6}}}		
\put (54.5,37){\textcolor{cyan}{\circle* {6}}}		
\put (69.5,37){\textcolor{cyan}{\circle* {6}}}		
\put (54.5,7){\textcolor{cyan}{\circle* {6}}}		
\thicklines
\put (22.5,116){\textcolor{magenta}{\circle {8}}}	
\put (37.5,101){\textcolor{magenta}{\circle {8}}}	
\put (54.5,116){\textcolor{magenta}{\circle {8}}}	
\put (101.5,101){\textcolor{magenta}{\circle {8}}}	
\put (22.5,84){\textcolor{magenta}{\circle {8}}}		
\put (84.5,69){\textcolor{magenta}{\circle {8}}}		
\put (69.5,54){\textcolor{magenta}{\circle {8}}}		
\put (37.5,37){\textcolor{magenta}{\circle {8}}}		
\put (131.5,22){\textcolor{magenta}{\circle {8}}}	
\put (116.5,7){\textcolor{magenta}{\circle {8}}}		
\put (84.5,116){\textcolor{magenta}{\circle {8}}}	
\put (84.5,101){\textcolor{magenta}{\circle {8}}}	
\put (116.5,116){\textcolor{magenta}{\circle {8}}}	
\put (116.5,101){\textcolor{magenta}{\circle {8}}}	
\put (22.5,54){\textcolor{magenta}{\circle {8}}}		
\put (37.5,54){\textcolor{magenta}{\circle {8}}}		
\put (84.5,54){\textcolor{magenta}{\circle {8}}}		
\put (22.5,22){\textcolor{magenta}{\circle {8}}}		
\put (37.5,22){\textcolor{magenta}{\circle {8}}}		
\put (116.5,22){\textcolor{magenta}{\circle {8}}}	
\put (84.5,131){\textcolor{green}{\circle {5}}}	
\put (116.5,131){\textcolor{green}{\circle {5}}}	
\put (7.5,54){\textcolor{green}{\circle {5}}}		
\put (54.5,84){\textcolor{green}{\circle {5}}}		
\put (116.5,69){\textcolor{green}{\circle {5}}}	
\put (131.5,54){\textcolor{green}{\circle {5}}}	
\put (7.5,22){\textcolor{green}{\circle {5}}}		
\put (69.5,22){\textcolor{green}{\circle {5}}}		
\put (84.5,7){\textcolor{green}{\circle {5}}}		
\put (101.5,37){\textcolor{green}{\circle {5}}}	
\put (69.5,84){\textcolor{green}{\circle {5}}}		
\put (54.5,69){\textcolor{green}{\circle {5}}}		
\put (69.5,69){\textcolor{green}{\circle {5}}}		
\put (116.5,84){\textcolor{green}{\circle {5}}}	
\put (101.5,54){\textcolor{green}{\circle {5}}}	
\put (84.5,37){\textcolor{green}{\circle {5}}}		
\put (54.5,22){\textcolor{green}{\circle {5}}}		
\put (131.5,37){\textcolor{green}{\circle {5}}}	
\put (101.5,7){\textcolor{green}{\circle {5}}}		
\put (131.5,7){\textcolor{green}{\circle {5}}}		
\end{picture}
\caption{The partition of $AG(4,3)$ with anchor point $a$ (asterisk), $C$ (solid black circles), $C'$ (small solid blue circles), $M_1$ (large red empty circles), and $M_2$ (small empty green circles). } 
\label{F:Partition}
\end{figure}
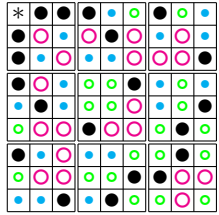

We can use either $M_1$ or $M_2$ to find the partition of $C'$ into 2 demicaps that corresponds with $C$. We will use $M_1=D_{1}\cup D_{1}'$.  From Corollary~\ref{C:demicapcompletecounts}, there are 40 points that complete one line with the points in $D_{1}$; 10 of those points are in $C'$, and they are a demicap. 
The complement of those points in $C'$ is, of course, a demicap.  
$C'$, the demicap $D_1$ and the two maximal caps whose union is $C'$ are shown in Figure~\ref{F:decompofC'}; they provide the decomposition of $C'$ into demicaps that corresponds to $C$.

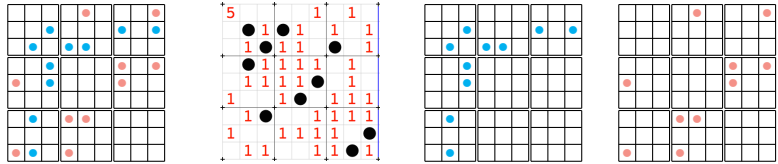
\begin{figure}[htbp]
\begin{center}
\begin{picture}(70,70)\setlength{\unitlength}{.015cm}		%
\usebox{\gridsm}
\put (37.5,116){\textcolor{ProcessBlue}{\circle* {8}}}	
\put (22.5,101){\textcolor{ProcessBlue}{\circle* {8}}}	
\put (69.5,131){\textcolor{Salmon}{\circle* {8}}}	
\put (54.5,101){\textcolor{ProcessBlue}{\circle* {8}}}	
\put (69.5,101){\textcolor{ProcessBlue}{\circle* {8}}}	
\put (131.5,131){\textcolor{Salmon}{\circle* {8}}}	
\put (101.5,116){\textcolor{ProcessBlue}{\circle* {8}}}	
\put (131.5,116){\textcolor{ProcessBlue}{\circle* {8}}}	
\put (37.5,84){\textcolor{ProcessBlue}{\circle* {8}}}		
\put (7.5,69){\textcolor{Salmon}{\circle* {8}}}		
\put (37.5,69){\textcolor{ProcessBlue}{\circle* {8}}}		
\put (101.5,84){\textcolor{Salmon}{\circle* {8}}}	
\put (131.5,84){\textcolor{Salmon}{\circle* {8}}}	
\put (101.5,69){\textcolor{Salmon}{\circle* {8}}}	
\put (22.5,37){\textcolor{ProcessBlue}{\circle* {8}}}		
\put (7.5,7){\textcolor{Salmon}{\circle* {8}}}		
\put (22.5,7){\textcolor{ProcessBlue}{\circle* {8}}}		
\put (54.5,37){\textcolor{Salmon}{\circle* {8}}}		
\put (69.5,37){\textcolor{Salmon}{\circle* {8}}}		
\put (54.5,7){\textcolor{Salmon}{\circle* {8}}}		
\end{picture} \hspace{.1in}
\includegraphics[width=.83in]{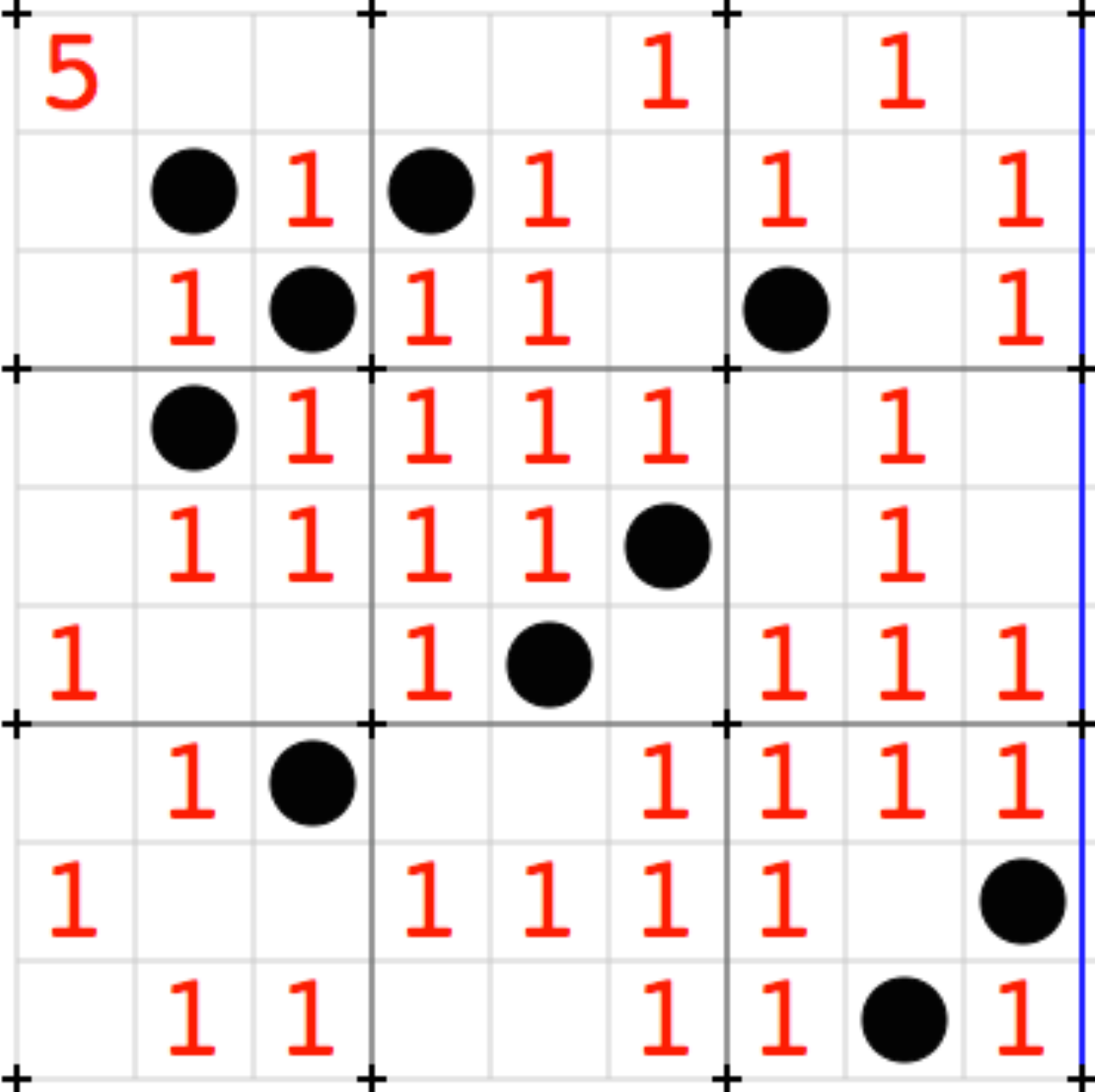}
\hspace{.1in}
\begin{picture}(70,70)\setlength{\unitlength}{.015cm}		%
\usebox{\gridsm}
\put (37.5,116){\textcolor{ProcessBlue}{\circle* {8}}}	
\put (22.5,101){\textcolor{ProcessBlue}{\circle* {8}}}	
\put (54.5,101){\textcolor{ProcessBlue}{\circle* {8}}}	
\put (69.5,101){\textcolor{ProcessBlue}{\circle* {8}}}	
\put (101.5,116){\textcolor{ProcessBlue}{\circle* {8}}}	
\put (131.5,116){\textcolor{ProcessBlue}{\circle* {8}}}	
\put (37.5,84){\textcolor{ProcessBlue}{\circle* {8}}}		
\put (37.5,69){\textcolor{ProcessBlue}{\circle* {8}}}		
\put (22.5,37){\textcolor{ProcessBlue}{\circle* {8}}}		
\put (22.5,7){\textcolor{ProcessBlue}{\circle* {8}}}		
\end{picture}
\begin{picture}(70,70)\setlength{\unitlength}{.015cm}		%
\usebox{\gridsm}
\put (69.5,131){\textcolor{Salmon}{\circle* {8}}}	
\put (131.5,131){\textcolor{Salmon}{\circle* {8}}}	
\put (7.5,69){\textcolor{Salmon}{\circle* {8}}}		
\put (101.5,84){\textcolor{Salmon}{\circle* {8}}}	
\put (131.5,84){\textcolor{Salmon}{\circle* {8}}}	
\put (101.5,69){\textcolor{Salmon}{\circle* {8}}}	
\put (7.5,7){\textcolor{Salmon}{\circle* {8}}}		
\put (54.5,37){\textcolor{Salmon}{\circle* {8}}}		
\put (69.5,37){\textcolor{Salmon}{\circle* {8}}}		
\put (54.5,7){\textcolor{Salmon}{\circle* {8}}}		
\end{picture}
\caption{Left: $C'$. Middle left: $D_1$ and the points that complete a line with two points from $D_1$. Middle right: the demicap of points in $C'$ that complete a line with points in $D_1$. Right: the demicap in $C'$ that completes no lines with points in $D_1$.}
\label{F:decompofC'}
\end{center}
\end{figure}

We summarize these results in a proposition.  As before, because all of what we did is preserved under affine transformations, the example we used provides the justification that this procedure will produce the same results  for any other maximal cap.

\begin{Prop}\label{P:gettheptn}
Let $C$ be a maximal cap in $AG(4,3)$, and let $C=D\cup D'$ be a decomposition of $C$ into two disjoint demicaps. Let $C'$ be the  maximal cap in exactly one partition with $C$ determined by $D$ and $D'$ via the correspondence given in Theorem~\ref{T:demicapcorresp}. Then:
\begin{enumerate}
\item Let $S_1$ be the 20 points in $AG(4,3)$ that complete a line with a pair of points from $D$ that do not complete a line with a pair of points from $D'$.  Similarly, let $S_2$ be the 20 points in $AG(4,3)$ that complete a line with a pair of points from $D'$ that do not complete a line with a pair of points from $D$.  Both $S_1$ and $S_2$ can be partitioned into a pair of disjoint demicaps in six ways. 
\item Exactly two of the demicaps $D_1$ and $D_2$ in $S_1$ can be matched with a demicap $D_1'$ and $D_2'$ respectively, in $S_2$ so that the union of those demicaps is a maximal cap: $M_1=D_1\cup D_1'$ and $M_2 = D_2 \cup D_2'$. $D_1$ and $D_2$ are a complementary pair in $S_1$, as are $D_1'$ and $D_2'$ in $S_2$.
\item $\{ \{a\},C,C',M_1,M_2\}$ is a partition of $AG(4,3)$ into two 1-completable pairs $\{C,C'\}$ and $\{M_1,M_2\}$.
\item The demicap partitions given in (2) for $M_1$ and $M_2$ determines the other maximal cap via the correspondence from Theorem~\ref{T:demicapcorresp}.
\item Let $D_1$ be one of the demicaps in $M_1$ as in (2). If $D''$ is the set of points in $C'$  that complete 1 line with a pair of points  from $D_1$, then $D''$ is a demicap.  If $D'''=C'\backslash D''$, then  $C' = D'' \cup D'''$ is the demicap decomposition of $C'$ that corresponds to $C$.
\end{enumerate}
\end{Prop}

We now have the mechanism that provides a direct bijection between the demicap decompositions of $C$ and the 36 maximal caps that make a 1-completable pair with $C$.  Even more, we can use the same pair of demicaps to identify all the maximal caps in the unique partition containing $C$ and $C'$, and we can also find the demicap decomposition of each of the maximal caps in the partition that corresponds to the other cap making a 1-completable pair in that partition.

We can now find all 36 maximal caps that are in a unique partition with $C$, by finding all the demicaps contained in $C$ and using the procedures described in this section.  However, in the next section, we introduce another way to find all 36 maximal caps in a 1-completable pair with $C$, and it turns out to be remarkably simple.


\section{The 36 maximal caps that are in one partition with $C$.}\label{S:36caps}

Our goal now is to find all 36 maximal caps that are in exactly one partition with $C$.  As seen above, we could do this by taking each of the pairs of disjoint demicaps whose union is $C$ and following the procedure in Theorem~\ref{T:demicapcorresp}, but there is a more illuminating way to find them.

\subsection{Finding the 36 maximal caps in exactly one partition with $C$.}

Start with a maximal cap $C$, two demicaps whose union is $C$, and the maximal cap $C'$ which makes a 1-completable pair with $C$, using the procedure in Theorem~\ref{T:demicapcorresp}. Using the procedure in Proposition~\ref{P:gettheptn}, find the two demicaps $C_1$ and $R_1$ (the choice for these names will become clear soon) whose union is $C'$ and which correspond to $C$ as in Theorem~\ref{T:demicapcorresp}.  
 
First, find the six maximal caps that $C_1$ is in.  
Notably, each of those maximal caps is in a unique partition of $AG(4,3)$ with $C$. 
Find the complement of $C_1$ in each of those maximal caps, one of which will be $R_1$.  Those six are pictured in Figure~\ref{F:Rdemicaps}; the one in the upper left is $R_1$. We will call these $R_1,R_2,\dots,R_6$.

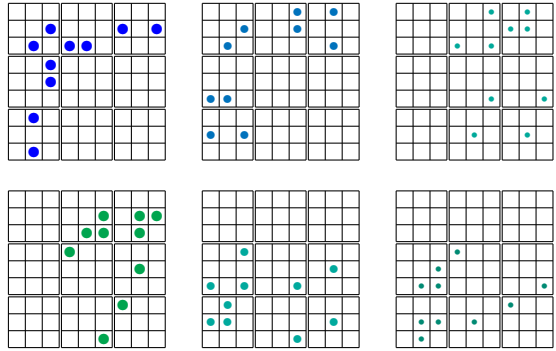
\begin{figure}[h]
\begin{picture}(70,70)\setlength{\unitlength}{.015cm}		%
\usebox{\gridsm}
\put (37.5,116){\textcolor{blue}{\circle* {10}}}
\put (22.5,101){\textcolor{blue}{\circle* {10}}}
\put (54.5,101){\textcolor{blue}{\circle* {10}}}
\put (69.5,101){\textcolor{blue}{\circle* {10}}}
\put (101.5,116){\textcolor{blue}{\circle* {10}}}
\put (131.5,116){\textcolor{blue}{\circle* {10}}}
\put (37.5,84){\textcolor{blue}{\circle* {10}}}
\put (37.5,69){\textcolor{blue}{\circle* {10}}}
\put (22.5,37){\textcolor{blue}{\circle* {10}}}
\put (22.5,7){\textcolor{blue}{\circle* {10}}}
\end{picture}
\begin{picture}(70,70)\setlength{\unitlength}{.015cm}		%
\usebox{\gridsm}
\put (37.5,116){\textcolor{RoyalBlue}{\circle* {7}}}
\put (22.5,101){\textcolor{RoyalBlue}{\circle* {7}}}
\put (84.5,131){\textcolor{RoyalBlue}{\circle* {7}}}
\put (84.5,116){\textcolor{RoyalBlue}{\circle* {7}}}
\put (116.5,131){\textcolor{RoyalBlue}{\circle* {7}}}
\put (116.5,101){\textcolor{RoyalBlue}{\circle* {7}}}
\put (7.5,54){\textcolor{RoyalBlue}{\circle* {7}}}	
\put (22.5,54){\textcolor{RoyalBlue}{\circle* {7}}}
\put (7.5,22){\textcolor{RoyalBlue}{\circle* {7}}}	
\put (37.5,22){\textcolor{RoyalBlue}{\circle* {7}}}
\end{picture}
\begin{picture}(70,70)\setlength{\unitlength}{.015cm}		%
\usebox{\gridsm}
\put (84.5,131){\textcolor{JungleGreen}{\circle* {4}}}
\put (54.5,101){\textcolor{JungleGreen}{\circle* {4}}}
\put (84.5,101){\textcolor{JungleGreen}{\circle* {4}}}
\put (116.5,131){\textcolor{JungleGreen}{\circle* {4}}}
\put (101.5,116){\textcolor{JungleGreen}{\circle* {4}}}
\put (116.5,116){\textcolor{JungleGreen}{\circle* {4}}}
\put (84.5,54){\textcolor{JungleGreen}{\circle* {4}}}
\put (131.5,54){\textcolor{JungleGreen}{\circle* {4}}}
\put (69.5,22){\textcolor{JungleGreen}{\circle* {4}}}
\put (116.5,22){\textcolor{JungleGreen}{\circle* {4}}}
\end{picture} \\
\begin{picture}(70,70)\setlength{\unitlength}{.015cm}		%
\usebox{\gridsm}
\put (84.5,116){\textcolor{Green}{\circle* {10}}}	
\put (69.5,101){\textcolor{Green}{\circle* {10}}}	
\put (84.5,101){\textcolor{Green}{\circle* {10}}}	
\put (116.5,116){\textcolor{Green}{\circle* {10}}}	
\put (131.5,116){\textcolor{Green}{\circle* {10}}}	
\put (116.5,101){\textcolor{Green}{\circle* {10}}}	
\put (54.5,84){\textcolor{Green}{\circle* {10}}}	
\put (116.5,69){\textcolor{Green}{\circle* {10}}}	
\put (84.5,7){\textcolor{Green}{\circle* {10}}}		
\put (101.5,37){\textcolor{Green}{\circle* {10}}}	
\end{picture}
\begin{picture}(70,70)\setlength{\unitlength}{.015cm}		%
\usebox{\gridsm}
\put (37.5,84){\textcolor{Emerald}{\circle* {7}}}	
\put (7.5,54){\textcolor{Emerald}{\circle* {7}}}	    
\put (37.5,54){\textcolor{Emerald}{\circle* {7}}}	
\put (84.5,54){\textcolor{Emerald}{\circle* {7}}}	
\put (116.5,69){\textcolor{Emerald}{\circle* {7}}}	
\put (22.5,37){\textcolor{Emerald}{\circle* {7}}}	
\put (7.5,22){\textcolor{Emerald}{\circle* {7}}}	
\put (22.5,22){\textcolor{Emerald}{\circle* {7}}}	
\put (84.5,7){\textcolor{Emerald}{\circle* {7}}}	
\put (116.5,22){\textcolor{Emerald}{\circle* {7}}}	
\end{picture}
\begin{picture}(70,70)\setlength{\unitlength}{.015cm}		%
\usebox{\gridsm}
\put (37.5,69){\textcolor{PineGreen}{\circle* {4}}}	
\put (22.5,54){\textcolor{PineGreen}{\circle* {4}}}	
\put (37.5,54){\textcolor{PineGreen}{\circle* {4}}}	
\put (54.5,84){\textcolor{PineGreen}{\circle* {4}}}	
\put (131.5,54){\textcolor{PineGreen}{\circle* {4}}}	
\put (22.5,22){\textcolor{PineGreen}{\circle* {4}}}	
\put (37.5,22){\textcolor{PineGreen}{\circle* {4}}}	
\put (22.5,7){\textcolor{PineGreen}{\circle* {4}}}	    
\put (69.5,22){\textcolor{PineGreen}{\circle* {4}}}	
\put (101.5,37){\textcolor{PineGreen}{\circle* {4}}}	
\end{picture}
\caption{The 6 demicaps $R_1,R_2,\dots,R_6$ whose union with $C_1$ is a maximal cap.} 
\label{F:Rdemicaps}
\end{figure}

Next, find the six maximal caps that $R_1$ is in.  
Again, each of those is in a unique partition of $AG(4,3)$ with $C$.
Find the complement of $R_1$ in each of those maximal caps, one of which will be $C_1$.  Those six are pictured in Figure~\ref{F:Cdemicaps}; the one in the upper left is $C_1$. We will call these $C_1,C_2,\dots,C_6$.

\begin{figure}[h]
\begin{picture}(70,70)\setlength{\unitlength}{.015cm}		%
\usebox{\gridsm}
\thicklines
\put (69.5,131){\textcolor{Red}{\circle{10}}}	
\put (131.5,131){\textcolor{Red}{\circle{10}}}	
\put (7.5,69){\textcolor{Red}{\circle{10}}}	
\put (101.5,84){\textcolor{Red}{\circle{10}}}	
\put (131.5,84){\textcolor{Red}{\circle{10}}}	
\put (101.5,69){\textcolor{Red}{\circle{10}}}	
\put (7.5,7){\textcolor{Red}{\circle{10}}}		
\put (54.5,37){\textcolor{Red}{\circle{10}}}	
\put (69.5,37){\textcolor{Red}{\circle{10}}}	
\put (54.5,7){\textcolor{Red}{\circle{10}}}	
\end{picture}
\begin{picture}(70,70)\setlength{\unitlength}{.015cm}		%
\usebox{\gridsm}
\thicklines
\put (22.5,116){\textcolor{BrickRed}{\circle{7}}}	
\put (37.5,101){\textcolor{BrickRed}{\circle{7}}}	
\put (69.5,84){\textcolor{BrickRed}{\circle{7}}}	
\put (69.5,54){\textcolor{BrickRed}{\circle{7}}}	
\put (116.5,84){\textcolor{BrickRed}{\circle{7}}}	
\put (131.5,84){\textcolor{BrickRed}{\circle{7}}}	
\put (69.5,37){\textcolor{BrickRed}{\circle{7}}}	
\put (84.5,37){\textcolor{BrickRed}{\circle{7}}}	
\put (131.5,37){\textcolor{BrickRed}{\circle{7}}}	
\put (131.5,22){\textcolor{BrickRed}{\circle{7}}}	
\end{picture}
\begin{picture}(70,70)\setlength{\unitlength}{.015cm}		%
\usebox{\gridsm}
\thicklines
\put (22.5,116){\textcolor{Plum}{\circle{4}}}	
\put (37.5,101){\textcolor{Plum}{\circle{4}}}	
\put (54.5,69){\textcolor{Plum}{\circle{4}}}	
\put (84.5,69){\textcolor{Plum}{\circle{4}}}	
\put (101.5,69){\textcolor{Plum}{\circle{4}}}	
\put (101.5,54){\textcolor{Plum}{\circle{4}}}	
\put (54.5,22){\textcolor{Plum}{\circle{4}}}	
\put (54.5,7){\textcolor{Plum}{\circle{4}}}		
\put (101.5,7){\textcolor{Plum}{\circle{4}}}	
\put (116.5,7){\textcolor{Plum}{\circle{4}}}	
\end{picture} \\
\begin{picture}(70,70)\setlength{\unitlength}{.015cm}		%
\usebox{\gridsm}
\thicklines
\put (69.5,131){\textcolor{Thistle}{\circle{10}}}	
\put (131.5,131){\textcolor{Thistle}{\circle{10}}}		
\put (22.5,84){\textcolor{Thistle}{\circle{10}}}	
\put (69.5,84){\textcolor{Thistle}{\circle{10}}}		
\put (69.5,69){\textcolor{Thistle}{\circle{10}}}		
\put (84.5,69){\textcolor{Thistle}{\circle{10}}}		
\put (37.5,37){\textcolor{Thistle}{\circle{10}}}		
\put (131.5,37){\textcolor{Thistle}{\circle{10}}}		
\put (116.5,7){\textcolor{Thistle}{\circle{10}}}		
\put (131.5,7){\textcolor{Thistle}{\circle{10}}}		
\end{picture}
\begin{picture}(70,70)\setlength{\unitlength}{.015cm}		%
\usebox{\gridsm}
\thicklines
\put (54.5,116){\textcolor{Orange}{\circle{7}}}
\put (101.5,101){\textcolor{Orange}{\circle{7}}}	
\put (22.5,84){\textcolor{Orange}{\circle{7}}}	
\put (101.5,84){\textcolor{Orange}{\circle{7}}}
\put (116.5,84){\textcolor{Orange}{\circle{7}}}	
\put (101.5,54){\textcolor{Orange}{\circle{7}}}	
\put (37.5,37){\textcolor{Orange}{\circle{7}}}	
\put (54.5,37){\textcolor{Orange}{\circle{7}}}  
\put (84.5,37){\textcolor{Orange}{\circle{7}}}	
\put (54.5,22){\textcolor{Orange}{\circle{7}}}	
\end{picture}
\begin{picture}(70,70)\setlength{\unitlength}{.015cm}		%
\usebox{\gridsm}
\thicklines
\put (54.5,116){\textcolor{Magenta}{\circle{4}}}	
\put (101.5,101){\textcolor{Magenta}{\circle{4}}}	
\put (7.5,69){\textcolor{Magenta}{\circle{4}}}	
\put (54.5,69){\textcolor{Magenta}{\circle{4}}}	
\put (69.5,69){\textcolor{Magenta}{\circle{4}}}	
\put (69.5,54){\textcolor{Magenta}{\circle{4}}}	
\put (7.5,7){\textcolor{Magenta}{\circle{4}}}		
\put (131.5,22){\textcolor{Magenta}{\circle{4}}}	
\put (101.5,7){\textcolor{Magenta}{\circle{4}}}		
\put (131.5,7){\textcolor{Magenta}{\circle{4}}}	
\end{picture}
\caption{The 6 demicaps $C_1,C_2,\dots,C_6$ whose union with $R_1$ is a maximal cap.} 
\label{F:Cdemicaps}
\end{figure}
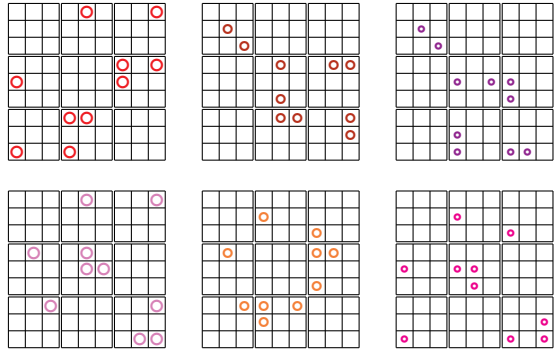

The next theorem describes the significance of the demicaps $C_i$ and $R_j$ that we have just found.  The authors have verified the theorem for the maximal cap $C$ we have been using throughout this paper.  Once again, because affine transformations preserve maximal caps, demicaps, and line incidences, our work with that maximal cap provides the same result for all maximal caps.

\begin{Thm}
Let $C$ and $C'$ be a 1-completable pair of maximal caps, and let $C_1$ and $R_1$ be the demicap decomposition of $C'$ that corresponds to $C$ via the correspondence in Theorem~\ref{T:demicapcorresp}. $C_1$ is in six maximal caps; let $R_1, \dots, R_6$ be the six complements of $C_1$ in those maximal caps.  Similarly, let $C_1,\dots,C_6$ denote the six complements of $R_1$ in the six maximal caps $R_1$ is in. Then for any $1\leq i,j\leq6$, $R_i\cup C_j$ is a maximal cap.  Further, the set of 36 maximal caps $\{R_i\cup C_j\},1\leq i,j\leq6$ is the set of 36 maximal caps that are in exactly one partition with $C$. 
\end{Thm}

This result is surprising: those 36 maximal caps are generated by just 12 demicaps in two sets of six.  This is why we used $R_i$ and $C_j$: we will place them in rows and columns.  Figure~\ref{F:Some1comps} shows a $2\times3$ section of the entire grid of 36, demonstrating the structure of these maximal caps.  

\begin{figure}[h]
\begin{picture}(70,70)\setlength{\unitlength}{.015cm}		%
\usebox{\gridsm}
\put (37.5,116){\textcolor{blue}{\circle* {10}}}
\put (22.5,101){\textcolor{blue}{\circle* {10}}}
\put (54.5,101){\textcolor{blue}{\circle* {10}}}
\put (69.5,101){\textcolor{blue}{\circle* {10}}}
\put (101.5,116){\textcolor{blue}{\circle* {10}}}
\put (131.5,116){\textcolor{blue}{\circle* {10}}}
\put (37.5,84){\textcolor{blue}{\circle* {10}}}
\put (37.5,69){\textcolor{blue}{\circle* {10}}}
\put (22.5,37){\textcolor{blue}{\circle* {10}}}
\put (22.5,7){\textcolor{blue}{\circle* {10}}}
\thicklines
\put (69.5,131){\textcolor{Red}{\circle {10}}}	
\put (131.5,131){\textcolor{Red}{\circle {10}}}	
\put (7.5,69){\textcolor{Red}{\circle {10}}}	
\put (101.5,84){\textcolor{Red}{\circle {10}}}	
\put (131.5,84){\textcolor{Red}{\circle {10}}}	
\put (101.5,69){\textcolor{Red}{\circle {10}}}	
\put (7.5,7){\textcolor{Red}{\circle {10}}}		
\put (54.5,37){\textcolor{Red}{\circle {10}}}	
\put (69.5,37){\textcolor{Red}{\circle {10}}}	
\put (54.5,7){\textcolor{Red}{\circle {10}}}	
\end{picture}
\begin{picture}(70,70)\setlength{\unitlength}{.015cm}		%
\usebox{\gridsm}
\put (37.5,116){\textcolor{blue}{\circle* {10}}}
\put (22.5,101){\textcolor{blue}{\circle* {10}}}
\put (54.5,101){\textcolor{blue}{\circle* {10}}}
\put (69.5,101){\textcolor{blue}{\circle* {10}}}
\put (101.5,116){\textcolor{blue}{\circle* {10}}}
\put (131.5,116){\textcolor{blue}{\circle* {10}}}
\put (37.5,84){\textcolor{blue}{\circle* {10}}}
\put (37.5,69){\textcolor{blue}{\circle* {10}}}
\put (22.5,37){\textcolor{blue}{\circle* {10}}}
\put (22.5,7){\textcolor{blue}{\circle* {10}}}
\thicklines
\put (22.5,116){\textcolor{BrickRed}{\circle {7}}}	
\put (37.5,101){\textcolor{BrickRed}{\circle {7}}}	
\put (69.5,84){\textcolor{BrickRed}{\circle {7}}}	
\put (69.5,54){\textcolor{BrickRed}{\circle {7}}}	
\put (116.5,84){\textcolor{BrickRed}{\circle {7}}}	
\put (131.5,84){\textcolor{BrickRed}{\circle {7}}}	
\put (69.5,37){\textcolor{BrickRed}{\circle {7}}}	
\put (84.5,37){\textcolor{BrickRed}{\circle {7}}}	
\put (131.5,37){\textcolor{BrickRed}{\circle {7}}}	
\put (131.5,22){\textcolor{BrickRed}{\circle {7}}}	
\end{picture}
\begin{picture}(70,70)\setlength{\unitlength}{.015cm}		%
\usebox{\gridsm}
\put (37.5,116){\textcolor{blue}{\circle* {10}}}
\put (22.5,101){\textcolor{blue}{\circle* {10}}}
\put (54.5,101){\textcolor{blue}{\circle* {10}}}
\put (69.5,101){\textcolor{blue}{\circle* {10}}}
\put (101.5,116){\textcolor{blue}{\circle* {10}}}
\put (131.5,116){\textcolor{blue}{\circle* {10}}}
\put (37.5,84){\textcolor{blue}{\circle* {10}}}
\put (37.5,69){\textcolor{blue}{\circle* {10}}}
\put (22.5,37){\textcolor{blue}{\circle* {10}}}
\put (22.5,7){\textcolor{blue}{\circle* {10}}}
\thicklines
\put (22.5,116){\textcolor{Plum}{\circle {4}}}	
\put (37.5,101){\textcolor{Plum}{\circle {4}}}	
\put (54.5,69){\textcolor{Plum}{\circle {4}}}	
\put (84.5,69){\textcolor{Plum}{\circle {4}}}	
\put (101.5,69){\textcolor{Plum}{\circle {4}}}	
\put (101.5,54){\textcolor{Plum}{\circle {4}}}	
\put (54.5,22){\textcolor{Plum}{\circle {4}}}	
\put (54.5,7){\textcolor{Plum}{\circle {4}}}		
\put (101.5,7){\textcolor{Plum}{\circle {4}}}	
\put (116.5,7){\textcolor{Plum}{\circle {4}}}	
\end{picture} \\
\begin{picture}(70,70)\setlength{\unitlength}{.015cm}		%
\usebox{\gridsm}
\put (37.5,116){\textcolor{RoyalBlue}{\circle* {7}}}
\put (22.5,101){\textcolor{RoyalBlue}{\circle* {7}}}
\put (84.5,131){\textcolor{RoyalBlue}{\circle* {7}}}
\put (84.5,116){\textcolor{RoyalBlue}{\circle* {7}}}
\put (116.5,131){\textcolor{RoyalBlue}{\circle* {7}}}
\put (116.5,101){\textcolor{RoyalBlue}{\circle* {7}}}
\put (7.5,54){\textcolor{RoyalBlue}{\circle* {7}}}	
\put (22.5,54){\textcolor{RoyalBlue}{\circle* {7}}}
\put (7.5,22){\textcolor{RoyalBlue}{\circle* {7}}}	
\put (37.5,22){\textcolor{RoyalBlue}{\circle* {7}}}
\thicklines
\put (69.5,131){\textcolor{Red}{\circle {10}}}	
\put (131.5,131){\textcolor{Red}{\circle {10}}}	
\put (7.5,69){\textcolor{Red}{\circle {10}}}	
\put (101.5,84){\textcolor{Red}{\circle {10}}}	
\put (131.5,84){\textcolor{Red}{\circle {10}}}	
\put (101.5,69){\textcolor{Red}{\circle {10}}}	
\put (7.5,7){\textcolor{Red}{\circle {10}}}		
\put (54.5,37){\textcolor{Red}{\circle {10}}}	
\put (69.5,37){\textcolor{Red}{\circle {10}}}	
\put (54.5,7){\textcolor{Red}{\circle {10}}}	
\end{picture}
\begin{picture}(70,70)\setlength{\unitlength}{.015cm}		%
\usebox{\gridsm}
\put (37.5,116){\textcolor{RoyalBlue}{\circle* {7}}}
\put (22.5,101){\textcolor{RoyalBlue}{\circle* {7}}}
\put (84.5,131){\textcolor{RoyalBlue}{\circle* {7}}}
\put (84.5,116){\textcolor{RoyalBlue}{\circle* {7}}}
\put (116.5,131){\textcolor{RoyalBlue}{\circle* {7}}}
\put (116.5,101){\textcolor{RoyalBlue}{\circle* {7}}}
\put (7.5,54){\textcolor{RoyalBlue}{\circle* {7}}}	
\put (22.5,54){\textcolor{RoyalBlue}{\circle* {7}}}
\put (7.5,22){\textcolor{RoyalBlue}{\circle* {7}}}	
\put (37.5,22){\textcolor{RoyalBlue}{\circle* {7}}}
\thicklines
\put (22.5,116){\textcolor{BrickRed}{\circle {7}}}	
\put (37.5,101){\textcolor{BrickRed}{\circle {7}}}	
\put (69.5,84){\textcolor{BrickRed}{\circle {7}}}	
\put (69.5,54){\textcolor{BrickRed}{\circle {7}}}	
\put (116.5,84){\textcolor{BrickRed}{\circle {7}}}	
\put (131.5,84){\textcolor{BrickRed}{\circle {7}}}	
\put (69.5,37){\textcolor{BrickRed}{\circle {7}}}	
\put (84.5,37){\textcolor{BrickRed}{\circle {7}}}	
\put (131.5,37){\textcolor{BrickRed}{\circle {7}}}	
\put (131.5,22){\textcolor{BrickRed}{\circle {7}}}	
\end{picture}
\begin{picture}(70,70)\setlength{\unitlength}{.015cm}		%
\usebox{\gridsm}
\put (37.5,116){\textcolor{RoyalBlue}{\circle* {7}}}
\put (22.5,101){\textcolor{RoyalBlue}{\circle* {7}}}
\put (84.5,131){\textcolor{RoyalBlue}{\circle* {7}}}
\put (84.5,116){\textcolor{RoyalBlue}{\circle* {7}}}
\put (116.5,131){\textcolor{RoyalBlue}{\circle* {7}}}
\put (116.5,101){\textcolor{RoyalBlue}{\circle* {7}}}
\put (7.5,54){\textcolor{RoyalBlue}{\circle* {7}}}	
\put (22.5,54){\textcolor{RoyalBlue}{\circle* {7}}}
\put (7.5,22){\textcolor{RoyalBlue}{\circle* {7}}}	
\put (37.5,22){\textcolor{RoyalBlue}{\circle* {7}}}
\thicklines
\put (22.5,116){\circle{4}}	
\put (37.5,101){\circle{4}}	
\put (54.5,69){\circle{4}}	
\put (84.5,69){\circle{4}}	
\put (101.5,69){\circle{4}}	
\put (101.5,54){\circle{4}}	
\put (54.5,22){\circle{4}}	
\put (54.5,7){\circle{4}}		
\put (101.5,7){\circle{4}}	
\put (116.5,7){\circle{4}}	
\end{picture}
\caption{A $2\times3$ portion of the collection of maximal caps that are in exactly one partition with $C$, enlarged to show detail.} 
\label{F:Some1comps}
\end{figure}
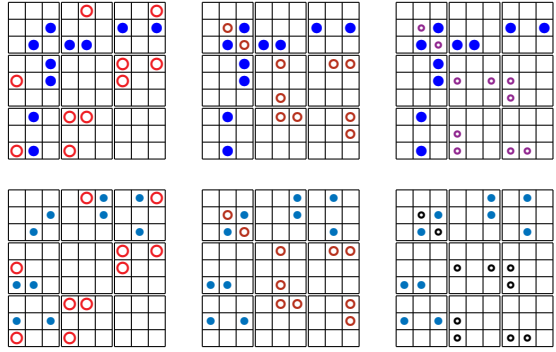

The entire collection of 36 maximal caps that make a 1-completable pair with $C$  are shown in Figure~\ref{F:The36}. One last feature is worth noting:  Every point in $AG(4,3)$ other than $a$ and the points in  $C$ is either in two of the demicaps $R_i$ or in two of the demicaps $C_j$.  Thus, every point other than $a$ and the points in $C$ is in 12 maximal caps that make a 1-completable pair with $C$.

        
        \begin{figure}
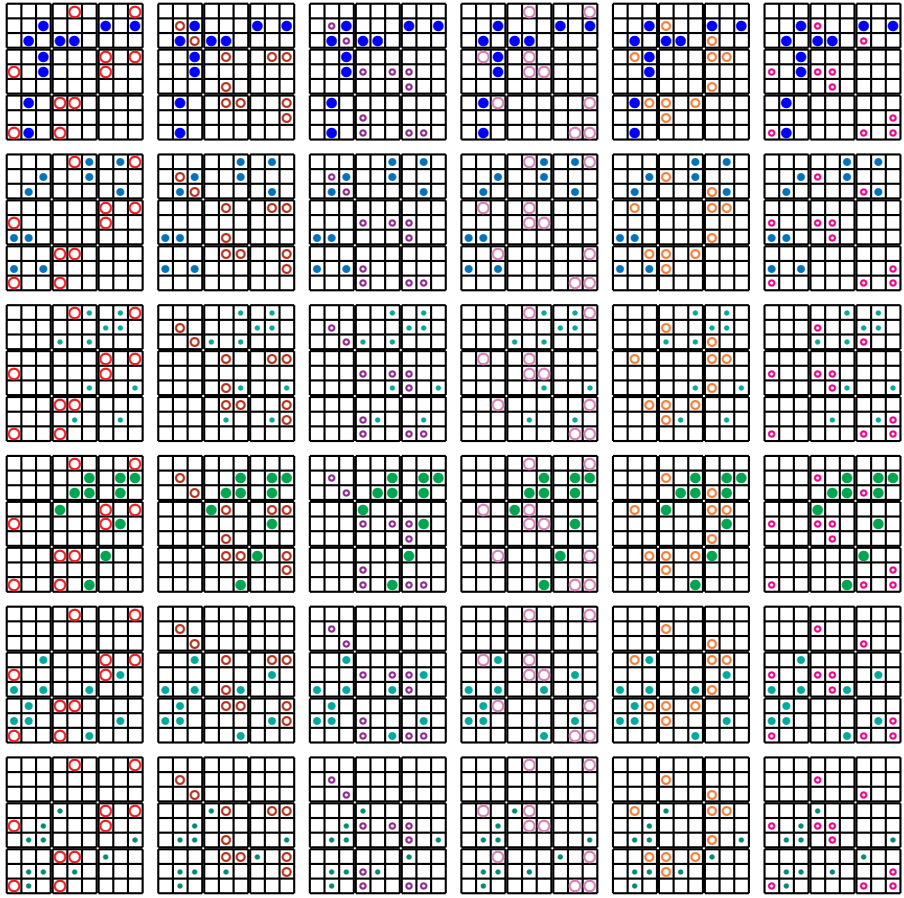

        \centering
        \begin{subfigure}[b]{0.15\textwidth}

        \end{subfigure}
        \caption{The 36 maximal caps in exactly one partition with the maximal cap $C$, shown in Figure~\ref{F:AG43cap}.}
        \label{F:The36}
        \end{figure}
        


\subsection{The action of the affine group and outer automorphisms of $S_6$.}
We finish this section with an exploration of the subgroup  $G$ of the affine group that fixes $C$ as a set. Because the affine group acting on $AG(4,3)$ is transitive on 1-completable pairs, any element of $G$ must permute the maximal caps that make a 1-completable pair with $C$.  

We start by fixing one particular maximal cap $C$.  As all maximal caps are affinely equivalent, without loss of generality we may choose the same cap $C$ we have been looking at throughout this paper (it is shown in Figure~\ref{F:AG43cap}).   As we saw in the previous subsections, the 36 maximal caps that make a 1-completable pair with $C$ are the unions of 12 demicaps: $\{R_i \cup C_j\},1\leq i,j \leq 6$. 

The affine group contains a subgroup $G$ consisting of the 2880 automorphisms of $AG(4,3)$ that fix $C$ as a set; this subgroup is discussed in \cite{MR3262358}.  Since $C$ is fixed as a set, the anchor point of $C$ is also fixed.  Since the anchor point for $C$ is $\vec0$, the subgroup of $G$ actually consists of linear transformations.  These linear transformations must permute the 36 maximal caps which are in exactly one partition with $C$.  Due to the structure of the 36 maximal caps that make a 1-completable pair with $C$, as outlined in the previous subsection, it must also be the case that the 12 special demicaps are also permuted by any automorphism in $G$.  

Suppose we have a linear transformation $\phi \in G$.  If $\phi(R_1) = R_i$, then each of the six maximal caps containing $R_1$ must have their images equal to one of the six maximal caps containing $R_i$. On the other hand, if  $\phi(R_1) = C_i$, then each of the six maximal caps containing $R_1$ must have their images equal to one of the six maximal caps containing $C_i$.  Because each of the columns intersects each of the rows in Figure~\ref{F:The36}, if $\phi(R_1)=R_i$, then $\phi(R_j) = R_{\sigma(j)}$ for some permutation $\sigma\in S_6$.  The same reasoning implies that if $\phi(R_1)=C_i$, then $\phi(R_j) = C_{\tau(j)}$ for some permutation $\tau\in S_6$. That means that the six rows and the six columns shown in Figure~\ref{F:The36} are permuted by $\phi$, in a particular way.  A given element of $G$ either permutes the $R_i$ and the $C_j$ separately, or it interchanges the $R_i$ and the $C_j$.

There are two particular transformations to note: the identity transformation, and the transformation that fixes $\vec0$ and reverses the two points that make a line with $\vec0$. Note that the second of these two is given by the matrix $-I$ (because the sum of the coordinates of the three points is $\vec0$, and one of the points is $\vec0$), so that transformation must be in the center of $G$; in addition, it fixes all lines through $\vec0$, so it fixes $C$, all maximal caps that make a 1-completable pair with $C$, and all demicaps with anchor $\vec0$ as sets. Thus $\{I, -I\}$ is a normal subgroup, and we can take the quotient group $G\slash\{I,-I\}$ to get 1440 affine transformations that  permute the demicaps and the maximal caps that make a 1-completable pair with $C$, in different ways. 

The following has been verified by using Mathematica \cite{Mathematica}. Of the 1440 affine transformations in the quotient group, 720 permute the set $\{R_i\},1\leq i\leq6$; that action is all of $S_6$.  For ease of exposition, we will identify the permutation on the demicaps $\{R_i\}$ by its action on the subscripts. We have identified permutations that act as $(12)$, $(13)$, \dots, $(16)$, and those permutations  generate all of $S_6$.  Similarly, those transformations permute the set of demicaps $\{C_j\},1\leq j\leq6$.  However, when a transformation acts on $\{R_i\}$ as (12), for example, it permutes the set $\{C_j\}$ as, for example, $(12)(34)(56)$ (the actual permutation will depend on how the set $\{C_j\}$ is numbered, of course, but it will always be a product of three disjoint transpositions).  Similarly, the transformation that acts as $(12)(34)(56)$ on $\{R_i\}$ will act on the $\{C_j\}$ as a transposition.  Further,  if the action of a transformation on  $\{R_i\}$ is a 6-cycle, the action of $g$ on  $\{C_j\}$ is a product of a 2-cycle and a 3-cycle.  This is precisely how outer automorphisms of $S_6$ work. 

The other 720 permutations will interchange the sets $\{R_i\}$ and $\{C_j\}$.  This action is a bit trickier,  because the permutation one gets will depend heavily on the ordering of the $R_i$ and the $C_j$, but with a carefully chosen numbering, the action is similar: the transformation that maps $R_1$ to $C_2$, $R_2$ to $C_1$, and $R_k$ to $C_k$ for $3\leq k \leq 6$, will map $C_1$ to $R_2$, $C_2$ to $R_1$, $C_3$ to $R_4$, $C_4$ to $R_3$, $C_5$ to $R_6$, and $C_6$ to $R_5$.  Again, we get the full action of $S_6$ with an outer automorphism twist.

\medskip

A closer look at the six demicaps $R_1,\dots,R_6$ in Figure~\ref{F:Rdemicaps} 
shows that each of the $R_i$ 
is the union of five of 15 $a$-lines. Further, if $i\neq j, 1\leq i,j \leq 6$,  $R_i \cap R_j$
is an $a$-line.  The same statements also hold for the the six demicaps  $C_1,\dots, C_6$ in Figure~\ref{F:Cdemicaps}, with a different set of 15 disjoint $a$-lines.  Readers familiar with the theory behind the outer automorphisms of $S_6$ might notice a similarity with synthemes and pentads, which have exactly the same structure: in the symmetric group $S_6$, there are 15 synthemes, each  the product of three disjoint 2-cycles.  There are exactly six pentads, namely sets of 5 synthemes with no 2-cycles in common.  The permutation of the pentads gives one of the standard descriptions of the outer automorphisms of $S_6$.  This can be found, for example, in Rotman's \emph{Introduction to Group Theory} \cite{MR1307623}.

\smallskip
It turns out that the action on the two block systems we've seen here is actually standard.  Inspired by this interesting situation, Jason Saied and Dantong Zhu (private communication) explored transitive actions of $S_6$ on sets of size 36.  They found that all transitive actions of $S_6$ on a set of 36 elements are canonically equivalent and take the same form as above. 
s

\section{Conclusion and future work}\label{S:future}

We have defined a new substructure of maximal caps in $AG(4,3)$ that provides a direct connection to the partitions of $AG(4,3)$ which consist of two 1-completable pairs of  maximal caps. Not only does a pair of demicaps whose union is a maximal cap $C$ provide a direct link to a particular maximal cap $C'$ in exactly one partition with the maximal cap that is the union of those two demicaps, but it is also possible to use that pair of demicaps to find the other 1-completable pair that is in the unique partition containing those two maximal caps.  In doing so, we have also identified the particular demicap decomposition of those two maximal caps that each give rise to the other.  Further, the demicap decomposition of $C'$ can be used to find the full collection of 36 maximal caps that are in exactly one partition with $C$.

There are still questions to explore, both with demicaps and with the partitions of $AG(4,3)$ into maximal caps and their associated anchor points.  For example, here are three questions.
\begin{itemize}
\item Given a partition that consists of two pairs of maximal caps, each of which is in only one partition, there are four mixed pairs of disjoint maximal caps (that is, one maximal cap from one pair and another maximal cap from another pair). Each of those mixed pairs is in 6 different partitions. In addition, any pair of disjoint maximal caps that are in six partitions are in only one partition where two pairs are each 1-completable; all the other partitions consists of two 2-completable pairs. Can demicaps be used to locate all partitions that a 6-completable pair is in? If not, is there some other way to find those partitions?
\item For a given maximal cap $C$, there are 72 
maximal caps $C'$ so that $\{C,C'\}$ are in 6 partitions.  Those 72 maximal caps pair up so that there are 36 pairs, one pair associated with each 1-completable pair that includes $C$.  Can our work here be used to distinguish those pairs?  
\item There are two equivalence classes of partitions under affine transformations. We have explored one of those equivalence classes.  Is there a similar substructure of maximal caps that helps to find the 2-completable pairs of maximal caps?  
\end{itemize}

Finally, it would be worth investigating whether a similar substructure of maximal caps exists in other dimensions.

\bibliographystyle{plain}
\bibliography{referencesset}

\end{document}